\documentclass[12pt, letterpaper]{amsart}
\usepackage{amsfonts,amsthm,mathtools}
\usepackage[all,cmtip]{xy}
\usepackage{fullpage}
\usepackage{amsmath}
\usepackage{amssymb}
\usepackage{enumerate}
\usepackage{tikz}
\usepackage[backref]{hyperref}
\usepackage{cleveref}
\usepackage{verbatim}
\usepackage{cancel}
\usepackage{float}
\usepackage{longtable}
\usepackage{listings}
\usepackage{makecell}
\usepackage{comment}

\newtheorem{lem}{Lemma}[section]
\newtheorem{thm}[lem]{Theorem}

\newtheorem{prop}[lem]{Proposition}
\newtheorem{cor}[lem]{Corollary}
\newtheorem{conj}[lem]{Conjecture}

\theoremstyle{definition}
\newtheorem{remark}[lem]{Remark}
\newtheorem{definition}[lem]{Definition}

\DeclareMathAlphabet{\curly}{U}{rsfs}{m}{n}

\newcommand{\gon}{\operatorname{gon}}

\newcommand{\Q}{\mathbb{Q}}
\newcommand{\C}{\mathbb{C}}

\newcommand{\F}{\mathbb{F}}

\newcommand{\PP}{{\mathbb P}}

\mathchardef\mhyphen="2D

\title{Tetragonal modular quotients of $X_0(N)$}

\author{\sc Petar Orli\'c}
\address{Petar Orli\'c \\
University of Zagreb\\  
Bijeni\v{c}ka Cesta 30 \\
10000 Zagreb\\
Croatia}
\email{petar.orlic@math.hr}

\begin{document}
\begin{abstract}
    Let $N$ be a positive integer. For every $d\mid N$ such that $(d,N/d)=1$ there exists an Atkin-Lehner involution $w_d$ of the modular curve $X_0(N)$. Let $B(N)$ be the group of all such involutions. In this paper we determine all $\C$ and $\Q$-tetragonal quotient curves $X_0(N)/W_N$, where $W_N\subseteq B(N)$ such that $4\leq|W_N|\leq 2^{\omega(N)-1}$, thus completing the classification of all $\C$-tetragonal quotients of $X_0(N)$ by Atkin-Lehner involutions.
\end{abstract}


\maketitle

\section{Introduction}
Let $k$ be a field and $C$ a smooth projective curve over $k$. The \textit{gonality} $\gon_k C$ of $C$ over $k$ is defined as the least degree of a non-constant morphism $f:C\rightarrow \PP^1$, or equivalently the least degree of a non-constant function $f\in k(C)$. There exists an upper bound for $\textup{gon}_k C$, linear in terms of the genus $g$.

\begin{prop}[{\cite[Proposition A.1.]{Poonen2007}}]\label{poonen}
    Let $X$ be a curve of genus $g$ over a field $k$.
    \begin{enumerate}[(i)]
        \item If $L$ is a field extension of $k$, then $\textup{gon}_L(X)\leq \textup{gon}_k(X)$.
        \item If $k$ is algebraically closed and $L$ is a field extension of $k$, then $\textup{gon}_L(X)=\textup{gon}_k(X)$.
        \item If $g\geq2$, then $\textup{gon}_k(X)\leq 2g-2$.
        \item If $g\geq2$ and $X(k)\neq\emptyset$, then $\textup{gon}_k(X)\leq g$.
        \item If $k$ is algebraically closed, then $\textup{gon}_k(X)\leq\frac{g+3}{2}$.
        \item If $\pi:X\to Y$ is a dominant $k$-rational map, then $\textup{gon}_k(X)\leq \deg \pi\cdot\textup{gon}_k(Y)$.
        \item If $\pi:X\to Y$ is a dominant $k$-rational map, then $\textup{gon}_k(X)\geq\textup{gon}_k(Y)$.
    \end{enumerate}
\end{prop}

When $C$ is a modular curve, there also exists a linear lower bound for the $\C$-gonality. This was first proved by Zograf \cite{Zograf1987}. The constant was later improved by Abramovich \cite{abramovich} and by Kim and Sarnak in Appendix 2 to \cite{Kim2002}.

Modular curves, especially $X_0(N)$ and $X_1(N)$, are very important objects in arithmetic geometry and their gonality has been the subject of extensive research. Ogg \cite{Ogg74} determined all hyperelliptic curves $X_0(N)$, Bars \cite{Bars99} determined all bielliptic curves $X_0(N)$, Hasegawa and Shimura determined all trigonal curves $X_0(N)$ over $\C$ and $\Q$, Jeon and Park determined all tetragonal curves $X_0(N)$ over $\C$, and Najman and Orlić \cite{NajmanOrlic22} determined all curves $X_0(N)$ with $\Q$-gonality equal to $4$, $5$, or $6$, and also determined the $\Q$ and $\C$-gonality for many other curves $X_0(N)$.

Modular quotients $X_0(N)/W_N$ are also important because points on them have a deep connection with $\Q$-curves, see \cite[Section 2]{NajmanVukorepa}. They also serve as a tool for studying the modular curve $X_0(N)$. Furumoto and Hasegawa \cite{FurumotoHasegawa1999} determined all hyperelliptic quotients $X_0(N)/W_N$, and Hasegawa and Shimura \cite{HasegawaShimura1999,HasegawaShimura2000,HasegawaShimura2006} determined all trigonal quotients $X_0(N)/W_N$ over $\C$. Jeon \cite{JEON2018319} determined all bielliptic curves $X_0^+(N)$, Bars, Gonzalez, and Kamel \cite{BARS2020380} determined all bielliptic quotients of $X_0(N)$ for squarefree levels $N$, Bars and Gonzalez determined all bielliptic curves $X_0^*(N)$, and Bars, Kamel, and Schweizer \cite{bars22biellipticquotients} determined all bielliptic quotients of $X_0(N)$ for non-squarefree levels $N$, completing the classification of bielliptic quotients.

The next logical step is to determine all tetragonal quotients of $X_0(N)$. All tetragonal quotients $X_0^{+d}(N)=X_0(N)/\left<w_d\right>$ over $\C$ and $\Q$ were determined in \cite{Orlic2023,Orlic2024}. All $\C$-tetragonal quotients $X_0^*(N)$ were determined in \cite{Orlic2025star}, as well as all $\Q$-tetragonal quotients with the exception of level $N=378$. The curve $X_0^*(378)$ is of genus $5$ and it is not yet determined whether its $\Q$-gonality is $4$ or $5$.

Here, we will do the same for the remaining curves $X_0(N)/W_N$, that is, the curves with $4\leq|W_N|\leq 2^{\omega(N)-1}$. We also determine all curves $X_0^*(N)$ of genus $4$ that are trigonal over $\Q$, thus completing the classification of all $\Q$-trigonal quotients of $X_0(N)$ (\cite{HasegawaShimura2006} does not study the $\Q$-gonality of genus $4$ quotients of $X_0(N)$).

Our main results are the following theorems.

\begin{thm}\label{trigonalthm}
    Let $W_N\subseteq B(N)$ such that $4\leq|W_N|\leq 2^{\omega(N)-1}$. The curve $X_0(N)/W_N$ is of genus $4$ and has $\Q$-gonality equal to $3$ if and only if \begin{align*}
        (N,W_N)\in\{&(130,\left<w_5,w_{13}\right>),(132,\left<w_4,w_{33}\right>),(154,\left<w_2,w_{77}\right>),(170,\left<w_5,w_{34}\right>),\\
        &(182,\left<w_2,w_{91}\right>),(186,\left<w_3,w_{62}\right>),(210,\left<w_2,w_{15},w_{21}\right>),\\
        &(255,\left<w_5,w_{51}\right>),(285,\left<w_3,w_{95}\right>),(286,\left<w_2,w_{143}\right>)\}.
    \end{align*}
\end{thm}

\begin{thm}\label{tetragonalthm}
    Let $W_N\subseteq B(N)$ such that $4\leq|W_N|\leq 2^{\omega(N)-1}$. The curve $X_0(N)/W_N$ has $\Q$-gonality equal to $4$ if and only if the pair $(N,W_N)$ is in \Cref{main:table} at the end of the paper. Furthermore, there are no curves $X_0(N)/W_N$ with $4\leq|W_N|\leq 2^{\omega(N)-1}$ that are $\C$-tetragonal, but have $\Q$-gonality greater than $4$.
\end{thm}

We use a variety of methods to prove these results. Each section roughly corresponds to one method. The paper is organized as follows: \begin{itemize}
    \item In \Cref{preliminaries_section} we eliminate all but finitely many levels $N$ for which the curve $X_0(N)/W_N$ is not $\Q$-tetragonal.
    \item In \Cref{Fpsection} we give lower bounds on the $\Q$-gonality via $\F_p$-gonality, either by counting points over $\F_{p^n}$ or computing the Riemann-Roch dimensions of degree $4$ effective $\F_p$-rational divisors. We study the $\F_p$-gonality only for curves of genus $g\geq10$ since for them we also get a bound on the $\C$-gonality, see Corollary \ref{towerthmcor}.
    \item In \Cref{betti_section} we determine the $\C$-gonality by computing Betti numbers $\beta_{i,j}$ for curves of genus $g\leq9$. We use them both to prove and disprove the existence of degree $4$ morphisms over $\C$ (and $\Q$) to $\PP^1$.
    \item In \Cref{CSsection} we use the Castelnuovo-Severi inequality to give a lower bound on the $\C$-gonality.
    \item In \Cref{rationalmapsection} we construct rational morphisms to $\PP^1$. The two main methods are finding a degree $2$ quotient map to another quotient or using \texttt{Magma} to find a morphism to $\PP^1$ of the desired degree.
    \item In \Cref{section_isomorphisms} we present isomorphisms between some quotients of $X_0(N)$ which enables us to determine their gonality more easily.
    \item In \Cref{thmproof_section} we combine the results from previous sections to obtain proofs of Theorems \ref{trigonalthm} and \ref{tetragonalthm}. For the reader's convenience, at the end of the paper after the proofs of these theorems we put \Cref{main:table} which contains all tetragonal quotient curves $X_0(N)/W_N$ with $4\leq|W_N|\leq2^{\omega(N)-1}$.
\end{itemize}

A significant number of results in this paper rely on \texttt{Magma} computations. The version of \texttt{Magma} used in the computations is V2.28-3. The codes that verify all computations in this paper can be found on
\begin{center}
    \url{https://github.com/orlic1/gonality_X0_star}.
\end{center}
All computations were performed on the Euler server at the Department of Mathematics, University of Zagreb with an Intel Xeon W-2133 CPU running at 3.60GHz and with
64 GB of RAM.

\section*{Acknowledgements}

Many thanks to Francesc Bars Cortina for useful discussions and permission to use \texttt{Magma} codes from his Github repository
\begin{center}
    \url{https://github.com/FrancescBars/Magma-functions-on-Quotient-Modular-Curves}.
\end{center} I am also grateful to Filip Najman and Maarten Derickx for their helpful comments and suggestions.

This paper was written during my visit to the Universitat Autònoma de Barcelona and I am grateful to the Department of Mathematics of Universitat Autònoma de Barcelona for its support and hospitality.

The author was supported by the Croatian Science Foundation under the project no. IP-2022-10-5008 and by the project “Implementation of cutting-edge research and its application as part of the Scientific Center of Excellence for Quantum and Complex Systems, and Representations of Lie Algebras“, Grant No. PK.1.1.02, co-financed by the European Union through the European Regional Development Fund - Competitiveness and Cohesion Programme 2021-2027.

\section{Preliminaries}\label{preliminaries_section}
Since all quotient curves $X_0(N)/W_N$ have at least one rational cusp, Proposition \ref{poonen} (iv) implies that their $\Q$-gonality is bounded from above by their genus. Moreover, due to the natural projection \[X_0(N)/W_N\to X_0^*(N)\] and Proposition \ref{poonen} (vii), when searching for tetragonal curves $X_0(N)/W_N$, we may restrict ourselves to the levels $N$ for which the curve $X_0^*(N)$ has $\C$-gonality at most $4$. All such levels $N$ are listed in \cite{Hasegawa1997, HasegawaShimura2000, Orlic2025star} and for all of them we have that $\omega(N)\leq4$. We may eliminate all levels $N$ with $\omega(N)\leq2$ since for these levels the only quotient curves by Atkin-Lehner involutions are $X_0^{+d}(N)$ for $d\mid\mid N$ and $X_0^*(N)$.

This leaves us with levels $N$ such that $\omega(N)=3,4$. If $\omega(N)=3$, there are $7$ quotient curves $X_0(N)/W_N$ with $|W_N|=4$. They are (for simplicity we suppose that $N=pqr$, $p\neq q\neq r\neq p$ are prime powers) \[W_N=\left<w_p,w_q\right>, \left<w_p,w_r\right>, \left<w_q,w_r\right>, \left<w_p,w_{qr}\right>, \left<w_q,w_{pr}\right>, \left<w_r,w_{pq}\right>, \left<w_{pq},w_{pr}\right>.\] If $\omega(N)=4$, there are $15$ quotient curves $X_0(N)/W_N$ with $|W_N|=8$. They are (for simplicity we suppose that $N=p_1p_2p_3p_4$, $p$-s are pairwise different prime powers) \begin{alignat*}{2}  
W_N=&\left<w_{p_i},w_{p_j},w_{p_k}\right> (4 \textup{ curves}), \ \ &&\left<w_{p_i},w_{p_j},w_{p_kp_l}\right> (6 \textup{ curves}),\\ &\left<w_{p_i},w_{p_jp_k}\right> (4 \textup{ curves}), \ \ &&\left<w_{p_1p_2},w_{p_1p_3},w_{p_1p_4}\right> (1\textup{ curve}).\end{alignat*}

There are also $35$ quotient curves $X_0(N)/W_N$ with $|W_N|=4$. They are \begin{alignat*}{2}
    W_N=&\left<w_{p_i},w_{p_j}\right> (6 \textup{ curves}), \ \ &&\left<w_{p_i},w_{p_jp_k}\right> (12\textup{ curves}),\\ &\left<w_{p_ip_j},w_{p_ip_k}\right> (4\textup{ curves}), \ \ &&\left<w_{p_ip_j},w_{p_kp_l}\right> (3\textup{ curves}),\\ &\left<w_{p_i},w_{p_ip_kp_l}\right> (6\textup{ curves}), \ \ &&\left<w_{p_i},w_{p_jp_kp_l}\right> (4\textup{ curves}).
\end{alignat*}

For many curves $X_0(N)/W_N$ we will be able to use $\Q$-gonality to give a lower bound on $\C$-gonality using the Tower theorem.

\begin{thm}[{The Tower Theorem, \cite[Proposition 2.4]{Poonen2007}}]\label{towerthm}
Let $C$ be a curve defined over a perfect field $k$ such that $C(k)\neq0$ and let $f:C\to \mathbb{P}^1$ be a non-constant morphism over $\overline{k}$ of degree $d$. Then there exists a curve $C'$ defined over $k$ and a non-constant morphism $C\to C'$ defined over $k$ of degree $d'$ dividing $d$ such that the genus of $C'$ is $\leq (\frac{d}{d'}-1)^2$.
\end{thm}

\begin{cor}\cite[Corollary 4.6. (ii)]{NajmanOrlic22}\label{towerthmcor}
    Let $C$ be a curve defined over $\Q$ with $\textup{gon}_\C (C)=4$ and $g(C)\geq10$ and such that $C(\Q)\neq\emptyset$. Then $\textup{gon}_\Q (C)=4$.
\end{cor}

Corollary \ref{towerthmcor} is applicable for all quotients of $X_0(N)$ since they all have at least $1$ rational cusp. The next result is an application of Corollary \ref{towerthmcor}. It deals with levels $N$ for which the curve $X_0^*(N)$ is tetragonal over $\C$, but not over $\Q$.
\begin{prop}\label{Ctetragonal_star_genus6}
    There are no $\C$-tetragonal curves $X_0(N)/W_N$ with $4\leq|W_N|\leq 2^{\omega(N)-1}$ for \[N\in\{506,561,590,609,615,690,858\}.\]
\end{prop}

\begin{proof}
    For these levels $N$ the curve $X_0^*(N)$ is of genus $6$, $\C$-gonality equal to $4$, and $\Q$-gonality greater than $4$ by \cite{Orlic2025star}. Therefore, the $\Q$-gonality of all curves $X_0(N)/W_N$ is also greater than $4$ by Proposition \ref{poonen} (vii).

    For $N\in\{506,561,590,609,615\}$ the genera of all $7$ quotients with $|W_N|=4$ are at least $10$. Thus, Corollary \ref{towerthmcor} implies that their $\C$-gonality is also greater than $4$.

    For $N\in\{690,858\}$ the genera of all $15$ quotients with $|W_N|=8$ are at least $10$. Thus, Corollary \ref{towerthmcor} implies that their $\C$-gonality is also greater than $4$ and Proposition \ref{poonen} (vii) implies the same for all $35$ quotients with $|W_N|=4$ (because for every such curve there is a degree $2$ rational map to a quotient by a group of order $8$).
\end{proof}

The remaining levels $N$ are those with $\omega(N)=3,4$ for which the $\Q$-gonality of $X_0^*(N)$ is $\leq4$ and $N=378$ (its $\Q$-gonality could be $4$ or $5$). Before moving on, we may eliminate levels \[N\in\{42,60,66,70,78,90,105\}\] because $g(X_0(N)/W_N)\leq3$ for all groups $W_N$ of order $4$.

\section{$\F_p$-gonality}\label{Fpsection}

In this section we use the results on the $\F_p$-gonality to get a lower bound on the $\Q$-gonality of quotient curves $X_0(N)/W_N$. $\F_p$-gonality is often much easier to compute than $\Q$-gonality and, since we have an inequality \[\textup{gon}_{\F_p}(C)\leq\textup{gon}_\Q(C)\] for primes $p$ of good reduction for curves $C/\Q$, it is very useful to give a lower bound on the $\Q$-gonality.

\begin{prop}\label{Fp_gonality}
The $\F_p$-gonality of the quotient curve $X_0(N)/W_N$ is at least $5$ for the following pairs $(N,W_N)$:

\begin{center}
\begin{longtable}{|c|c|c||c|c|c||c|c|c|}
\hline
\addtocounter{table}{-1}
$N$ & $W_N$ & $p$ & $N$ & $W_N$ & $p$ & $N$ & $W_N$ & $p$\\ 
    \hline
$210$ & $\left<w_2,w_{15}\right>$ & $11$ & $210$ & $\left<w_2,w_{11}\right>$ & $11$ & $210$ & $\left<w_{15},w_{42}\right>$ & $11$\\
\hline
$210$ & $\left<w_{21},w_{30}\right>$ & $11$ & $260$ & $\left<w_5,w_{13}\right>$ & $3$ & $260$ & $\left<w_{20},w_{52}\right>$ & $3$\\
\hline
$264$ & $\left<w_3,w_{11}\right>$ & $5$ & $264$ & $\left<w_8,w_{11}\right>$ & $5$ & $264$ & $\left<w_{24},w_{33}\right>$ & $5$\\
\hline
$280$ & $\left<w_5,w_7\right>$ & $3$ & $280$ & $\left<w_5,w_8\right>$ & $3$ & $280$ & $\left<w_8,w_{35}\right>$ & $3$\\
\hline
$300$ & $\left<w_3,w_{25}\right>$ & $7$ & $300$ & $\left<w_{12},w_{25}\right>$ & $7$ & $300$ & $\left<w_{12},w_{75}\right>$ & $7$\\
\hline
$300$ & $\left<w_3,w_{100}\right>$ & $7$ & $306$ & $\left<w_2,w_9\right>$ & $7$ & $306$ & $\left<w_2,w_{17}\right>$ & $5$\\
\hline
$306$ & $\left<w_9,w_{17}\right>$ & $5$ & $308$ & $\left<w_4,w_{11}\right>$ & $3$ & $308$ & $\left<w_7,w_{11}\right>$ & $3$\\
\hline
$308$ & $\left<w_{28},w_{44}\right>$ & $3$ & $322$ & $\left<w_2,w_7\right>$ & $5$ & $322$ & $\left<w_7,w_{46}\right>$ & $3$\\
\hline
$322$ & $\left<w_{14},w_{23}\right>$ & $3$ & $330$ & $\left<w_2,w_{11}\right>$ & $7$ & $330$ & $\left<w_5,w_{11}\right>$ & $7$\\
\hline
$330$ & $\left<w_2,w_{33}\right>$ & $7$ & $330$ & $\left<w_2,w_{55}\right>$ & $7$ & $330$ & $\left<w_3,w_{22}\right>$ & $7$\\
\hline
$330$ & $\left<w_5,w_{22}\right>$ & $7$ & $330$ & $\left<w_6,w_{11}\right>$ & $7$ & $330$ & $\left<w_{10},w_{11}\right>$ & $7$\\
\hline
$330$ & $\left<w_{22},w_{30}\right>$ & $7$ & $330$ & $\left<w_{30},w_{33}\right>$ & $7$ & $330$ & $\left<w_2,w_{165}\right>$ & $7$\\
\hline
$330$ & $\left<w_3,w_{110}\right>$ & $7$ & $336$ & $\left<w_3,w_{16}\right>$ & $5$ & $336$ & $\left<w_3,w_7\right>$ & $5$\\
\hline
$336$ & $\left<w_7,w_{16}\right>$ & $5$ & $336$ & $\left<w_{21},w_{48}\right>$ & $5$ & $336$ & $\left<w_7,w_{48}\right>$ & $5$\\
\hline
$336$ & $\left<w_3,w_{112}\right>$ & $5$ & $340$ & $\left<w_5,w_{17}\right>$ & $3$ & $340$ & $\left<w_{17},w_{20}\right>$ & $3$\\
\hline
$340$ & $\left<w_{20},w_{68}\right>$ & $3$ & $340$ & $\left<w_5,w_{68}\right>$ & $3$ & $342$ & $\left<w_9,w_{19}\right>$ & $5$\\
\hline
$342$ & $\left<w_9,w_{38}\right>$ & $5$ & $345$ & $\left<w_3,w_5\right>$ & $7$ & $345$ & $\left<w_5,w_{23}\right>$ & $2$\\
\hline
$345$ & $\left<w_3,w_{23}\right>$ & $2$ & $345$ & $\left<w_{15},w_{23}\right>$ & $2$ & $345$ & $\left<w_3,w_{115}\right>$ & $2$\\
\hline
$350$ & $\left<w_7,w_{50}\right>$ & $3$ & $350$ & $\left<w_{14},w_{25}\right>$ & $3$ & $354$ & $\left<w_2,w_{59}\right>$ & $5$\\
\hline
$354$ & $\left<w_3,w_{59}\right>$ & $5$ & $360$ & $\left<w_8,w_9\right>$ & $7$ & $360$ & $\left<w_5,w_8\right>$ & $7$\\
\hline
$360$ & $\left<w_5,w_9\right>$ & $7$ & $360$ & $\left<w_8,w_{45}\right>$ & $7$ & $364$ & $\left<w_7,w_{13}\right>$ & $3$\\
\hline
$364$ & $\left<w_{13},w_{28}\right>$ & $3$ & $364$ & $\left<w_{28},w_{152}\right>$ & $3$ & $364$ & $\left<w_7,w_{52}\right>$ & $3$\\
\hline
$366$ & $\left<w_3,w_{61}\right>$ & $5$ & $366$ & $\left<w_3,w_{122}\right>$ & $5$ & $366$ & $\left<w_2,w_{183}\right>$ & $5$\\
\hline
$370$ & $\left<w_2,w_{37}\right>$ & $3$ & $370$ & $\left<w_5,w_{74}\right>$ & $3$ & $370$ & $\left<w_2,w_{185}\right>$ & $3$\\
\hline
$372$ & $\left<w_3,w_{31}\right>$ & $5$ & $372$ & $\left<w_{12},w_{31}\right>$ & $5$ & $372$ & $\left<w_{12},w_{93}\right>$ & $5$\\
\hline
$372$ & $\left<w_3,w_{124}\right>$ & $5$ & $374$ & $\left<w_{11},w_{17}\right>$ & $3$ & $374$ & $\left<w_{11},w_{34}\right>$ & $3$\\
\hline
$378$ & $\left<w_2,w_{27}\right>$ & $5$ & $378$ & $\left<w_7,w_{27}\right>$ & $5$ & $378$ & $\left<w_{14},w_{27}\right>$ & $5$\\
\hline
$378$ & $\left<w_{14},w_{54}\right>$ & $5$ & $378$ & $\left<w_7,w_{54}\right>$ & $5$ & $378$ & $\left<w_2,w_{189}\right>$ & $5$\\
\hline
$385$ & $\left<w_5,w_7\right>$ & $3$ & $385$ & $\left<w_5,w_{11}\right>$ & $2$ & $385$ & $\left<w_5,w_{77}\right>$ & $2$\\
\hline
$390$ & $\left<w_3,w_{65}\right>$ & $11$ & $390$ & $\left<w_5,w_{39}\right>$ & $7$ & $390$ & $\left<w_6,w_{26}\right>$ & $7$\\
\hline
$390$ & $\left<w_{26},w_{30}\right>$ & $7$ & $390$ & $\left<w_2,w_{195}\right>$ & $7$ & $396$ & $\left<w_4,w_{11}\right>$ & $5$\\
\hline
$396$ & $\left<w_9,w_{11}\right>$ & $5$ & $396$ & $\left<w_9,w_{44}\right>$ & $5$ & $396$ & $\left<w_{36},w_{44}\right>$ & $5$\\
\hline
$402$ & $\left<w_2,w_{67}\right>$ & $7$ & $402$ & $\left<w_6,w_{134}\right>$ & $5$ & $402$ & $\left<w_3,w_{134}\right>$ & $5$\\
\hline
$402$ & $\left<w_2,w_{201}\right>$ & $5$ & $406$ & $\left<w_7,w_{29}\right>$ & $3$ & $406$ & $\left<w_{14},w_{29}\right>$ & $3$\\
\hline
$406$ & $\left<w_{14},w_{58}\right>$ & $3$ & $406$ & $\left<w_7,w_{58}\right>$ & $3$ & $406$ & $\left<w_2,w_{203}\right>$ & $3$\\
\hline
$410$ & $\left<w_2,w_{41}\right>$ & $3$ & $410$ & $\left<w_5,w_{41}\right>$ & $3$ & $410$ & $\left<w_{10},w_{41}\right>$ & $3$\\
\hline
$410$ & $\left<w_5,w_{82}\right>$ & $3$ & $410$ & $\left<w_2,w_{205}\right>$ & $3$ & $414$ & $\left<w_9,w_{23}\right>$ & $5$\\
\hline
$414$ & $\left<w_{18},w_{23}\right>$ & $5$ & $418$ & $\left<w_{11},w_{19}\right>$ & $3$ & $418$ & $\left<w_{22},w_{38}\right>$ & $3$\\
\hline
$418$ & $\left<w_2,w_{209}\right>$ & $3$ & $435$ & $\left<w_3,w_{29}\right>$ & $2$ & $435$ & $\left<w_5,w_{87}\right>$ & $2$\\
\hline
$435$ & $\left<w_3,w_{145}\right>$ & $2$ & $438$ & $\left<w_3,w_{146}\right>$ & $5$ & $438$ & $\left<w_6,w_{146}\right>$ & $5$\\
\hline
$442$ & $\left<w_2,w_{13}\right>$ & $3$ & $442$ & $\left<w_{13},w_{17}\right>$ & $3$ & $442$ & $\left<w_{17},w_{26}\right>$ & $3$\\
\hline
$442$ & $\left<w_{26},w_{34}\right>$ & $3$ & $442$ & $\left<w_2,w_{221}\right>$ & $3$ & $444$ & $\left<w_3,w_{37}\right>$ & $5$\\
\hline
$444$ & $\left<w_{12},w_{111}\right>$ & $5$ & $456$ & $\left<w_8,w_{19}\right>$ & $5$ & $456$ & $\left<w_{19},w_{24}\right>$ & $5$\\
\hline
$456$ & $\left<w_{24},w_{57}\right>$ & $5$ & $456$ & $\left<w_8,w_{57}\right>$ & $5$ & $456$ & $\left<w_3,w_{152}\right>$ & $5$\\
\hline
$470$ & $\left<w_{10},w_{47}\right>$ & $3$ & $494$ & $\left<w_{13},w_{38}\right>$ & $3$ & $504$ & $\left<w_9,w_{56}\right>$ & $5$\\
\hline
$520$ & $\left<w_8,w_{13}\right>$ & $7$ & $558$ & $\left<w_9,w_{31}\right>$ & $5$ & $558$ & $\left<w_{18},w_{62}\right>$ & $5$\\
\hline
$630$ & $\left<w_5,w_7,w_9\right>$ & $11$ & $630$ & $\left<w_2,w_5,w_{63}\right>$ & $11$ & $630$ & $\left<w_2,w_9,w_35\right>$ & $11$\\
\hline
$630$ & $\left<w_2,w_7,w_{45}\right>$ & $11$ & $630$ & $\left<w_5,w_9,w_{14}\right>$ & $11$ & $630$ & $\left<w_2,w_{35},w_{45}\right>$ & $11$\\
\hline
$630$ & $\left<w_9,w_{10},w_{14}\right>$ & $11$ & $660$ & $\left<w_3,w_5,w_{11}\right>$ & $7$ & $660$ & $\left<w_3,w_5,w_{44}\right>$ & $7$\\
\hline
$660$ & $\left<w_3,w_{11},w_{20}\right>$ & $7$ & $660$ & $\left<w_5,w_{11},w_{12}\right>$ & $7$ & $660$ & $\left<w_3,w_{20},w_{44}\right>$ & $7$\\
\hline
$660$ & $\left<w_5,w_{12},w_{33}\right>$ & $7$ & $660$ & $\left<w_{11},w_{12},w_{15}\right>$ & $7$ & $660$ & $\left<w_{12},w_{15},w_{33}\right>$ & $7$\\
\hline
$714$ & $\left<w_2,w_7,w_{17}\right>$ & $5$ & $714$ & $\left<w_3,w_7,w_{17}\right>$ & $5$ & $714$ & $\left<w_2,w_3,w_{119}\right>$ & $5$\\
\hline
$714$ & $\left<w_6,w_7,w_{17}\right>$ & $5$ & $714$ & $\left<w_2,w_{21},w_{51}\right>$ & $5$ & $714$ & $\left<w_3,w_{14},w_{34}\right>$ & $5$\\
\hline
$714$ & $\left<w_6,w_{14},w_{34}\right>$ & $7$ & $770$ & $\left<w_5,w_7,w_{22}\right>$ & $3$ & $780$ & $\left<w_3,w_{13},w_{20}\right>$ & $7$\\
\hline
$780$ & $\left<w_5,w_{12},w_{39}\right>$ & $7$ & $798$ & $\left<w_2,w_3,w_{133}\right>$ & $5$ & $798$ & $\left<w_2,w_{19},w_{21}\right>$ & $5$\\
\hline
$798$ & $\left<w_3,w_{14},w_{38}\right>$ & $5$ & $798$ & $\left<w_6,w_7,w_{38}\right>$ & $5$ & $798$ & $\left<w_6,w_{14},w_{19}\right>$ & $5$\\
\hline
$990$ & $\left<w_5,w_9,w_{11}\right>$ & $7$ & $990$ & $\left<w_2,w_5,w_{99}\right>$ & $7$ & $990$ & $\left<w_2,w_9,w_{55}\right>$ & $7$\\
\hline
$990$ & $\left<w_2,w_{11},w_{45}\right>$ & $7$ & $990$ & $\left<w_9,w_{10},w_{11}\right>$ & $7$ & & &\\
    \hline
\end{longtable}
\end{center}
\end{prop}

\begin{proof}
    Using \texttt{Magma}, we compute that there are no morphisms $X_0(N)/W_N\to\PP^1$ of degree $\leq4$ defined over $\F_p$. We are able to reduce the number of degree $4$ divisors that need to be checked by noting the following: If there exists a function $f:C\to\PP^1$ over $\F_p$ of a certain degree and if $\#C(\F_p)>d(p+1)$, then there exists $c\in \F_p$ such that the function $g(x):=\frac{1}{f(x)-c}$ has the same degree and its polar divisor contains at least $d$ $\F_p$-rational points, see \cite[Lemma 3.1]{NajmanOrlic22}.

    Therefore, if $\#(X_0(N)/W_N)(\F_p)>p+1$, it is enough to check divisors of the form $1+1+1+1$ and $1+1+2$, and if $\#(X_0(N)/W_N)(\F_p)>2(p+1)$, it is enough to just check divisors of the form $1+1+1+1$.
\end{proof}

The computations for some curves took several days to finish. These curves either have a high genus ($g\geq17$) or the prime number $p$ is large ($p\geq7$). As the genus increases, the number of divisors that need to be checked grows significantly. The same is true when $p$ grows because Hasse's theorem tells us that $\#C(\F_{p^k})\sim p^k$ for large $p$. See \cite[Section 8.1]{NajmanOrlic22} for more details. Here are some examples of curves with higher running time.

    \begin{center}
\begin{tabular}{|c|c|c|c|c|}
\hline
$N$ & $W_N$ & $p$ & $g$ & \textup{time}\\
    \hline

    $630$ & $\left<w_5,w_9,w_{14}\right>$ & $11$ & $13$ & $3$ \textup{ hours}\\
    $990$ & $\left<w_2,w_9,w_{55}\right>$ & $7$ & $19$ & $6$ \textup{ hours}\\
    $520$ & $\left<w_8,w_{13}\right>$ & $7$ & $17$ & $1$ \textup{ day}\\

    \hline
\end{tabular}
\end{center}

We can also get a lower bound on the $\Q$-gonality by counting $\F_{p^k}$ points on the curve without actually computing the $\F_p$-gonality.

\begin{lem}\cite[Lemma 3.5]{NajmanOrlic22}\label{Fp2points}
    Let $C/\Q$ be a curve, $p$ a prime of good reduction for $C$, and $q$ a power of $p$. Suppose $\#C(\F_q)>d(q+1)$. Then $\textup{gon}_\Q(C)>d$.
\end{lem}

\begin{prop}\label{prop_counting_Fpn_points}
    The $\F_p$-gonality of the quotient curve $X_0(N)/W_N$ is at least $5$ for the following pairs of $(N,W_N)$:

\begin{center}
\begin{longtable}{|c|c|c|c||c|c|c|c|}
\hline
\addtocounter{table}{-1}
$N$ & $W_N$ & $q$ & $\#(X_0(N)/W_N)(\F_q)$ & $N$ & $W_N$ & $q$ & $\#(X_0(N)/W_N)(\F_q)$\\
    \hline
$340$ & $\left<w_5,w_{17}\right>$ & $3$ & $18$ & $350$ & $\left<w_7,w_{25}\right>$ & $9$ & $44$\\
\hline
$410$ & $\left<w_5,w_{41}\right>$ & $9$ & $42$ & $418$ & $\left<w_2,w_{19}\right>$ & $9$ & $41$\\
\hline
$418$ & $\left<w_{19},w_{22}\right>$ & $9$ & $43$ & $435$ & $\left<w_5,w_{29}\right>$ & $4$ & $21$\\
\hline
$435$ & $\left<w_{15},w_{29}\right>$ & $4$ & $21$ & $440$ & $\left<w_5,w_{11}\right>$ & $9$ & $48$\\
\hline
$440$ & $\left<w_{11},w_{40}\right>$ & $9$ & $42$ & $440$ & $\left<w_{40},w_{55}\right>$ & $9$ & $50$\\
\hline
$440$ & $\left<w_8,w_{55}\right>$ & $9$ & $48$ & $460$ & $\left<w_5,w_{23}\right>$ & $9$ & $48$\\
\hline
$460$ & $\left<w_{20},w_{33}\right>$ & $9$ & $46$ & $460$ & $\left<w_{20},w_{92}\right>$ & $9$ & $46$\\
\hline
$460$ & $\left<w_5,w_{92}\right>$ & $9$ & $48$ & $465$ & $\left<w_5,w_{31}\right>$ & $4$ & $21$\\
\hline
$465$ & $\left<w_{15},w_{31}\right>$ & $4$ & $22$ & $465$ & $\left<w_{15},w_{93}\right>$ & $4$ & $22$\\
\hline
$465$ & $\left<w_3,w_{155}\right>$ & $4$ & $23$ & $470$ & $\left<w_{10},w_{94}\right>$ & $9$ & $47$\\
\hline
$470$ & $\left<w_5,w_{94}\right>$ & $9$ & $48$ & $470$ & $\left<w_2,w_{235}\right>$ & $9$ & $42$\\
\hline
$494$ & $\left<w_{26},w_{38}\right>$ & $9$ & $47$ & $494$ & $\left<w_2,w_{247}\right>$ & $9$ & $47$\\
\hline
$495$ & $\left<w_9,w_{11}\right>$ & $4$ & $24$ & $495$ & $\left<w_9,w_{55}\right>$ & $4$ & $22$\\
\hline
$495$ & $\left<w_{11},w_{45}\right>$ & $4$ & $23$ & $504$ & $\left<w_7,w_{9}\right>$ & $25$ & $112$\\
\hline
$504$ & $\left<w_{54},w_{63}\right>$ & $25$ & $116$ & $520$ & $\left<w_5,w_8\right>$ & $9$ & $56$\\
\hline
$520$ & $\left<w_{13},w_{40}\right>$ & $3$ & $18$ & $520$ & $\left<w_{40},w_{65}\right>$ & $9$ & $48$\\
\hline
$520$ & $\left<w_8,w_{65}\right>$ & $9$ & $56$ & $520$ & $\left<w_5,w_{104}\right>$ & $9$ & $48$\\
\hline
$528$ & $\left<w_{11},w_{16}\right>$ & $25$ & $110$ & $528$ & $\left<w_{16},w_{33}\right>$ & $25$ & $114$\\
\hline
$528$ & $\left<w_{11},w_{48}\right>$ & $25$ & $114$ & $528$ & $\left<w_{33},w_{48}\right>$ & $25$ & $110$\\
\hline
$528$ & $\left<w_3,w_{176}\right>$ & $25$ & $122$ & $532$ & $\left<w_7,w_{19}\right>$ & $9$ & $46$\\
\hline
$532$ & $\left<w_7,w_{76}\right>$ & $9$ & $48$ & $532$ & $\left<w_{19},w_{28}\right>$ & $9$ & $52$\\
\hline
$532$ & $\left<w_{28},w_{76}\right>$ & $3$ & $18$ & $555$ & $\left<w_{15},w_{111}\right>$ & $4$ & $25$\\
\hline
$555$ & $\left<w_5,w_{111}\right>$ & $4$ & $22$ & $555$ & $\left<w_3,w_{185}\right>$ & $4$ & $26$\\
\hline
$560$ & $\left<w_5,w_{16}\right>$ & $9$ & $58$ & $560$ & $\left<w_5,w_{112}\right>$ & $9$ & $54$\\
\hline
$560$ & $\left<w_{35},w_{112}\right>$ & $9$ & $54$ & $560$ & $\left<w_{16},w_{35}\right>$ & $9$ & $54$\\
\hline
$560$ & $\left<w_7,w_{80}\right>$ & $9$ & $54$ & $564$ & $\left<w_{12},w_{47}\right>$ & $25$ & $126$\\
\hline
$572$ & $\left<w_{11},w_{13}\right>$ & $9$ & $45$ & $572$ & $\left<w_{44},w_{52}\right>$ & $9$ & $47$\\
\hline
$574$ & $\left<w_7,w_{41}\right>$ & $9$ & $46$ & $574$ & $\left<w_{14},w_{82}\right>$ & $9$ & $52$\\
\hline
$574$ & $\left<w_2,w_{287}\right>$ & $9$ & $50$ & $585$ & $\left<w_5,w_9\right>$ & $4$ & $26$\\
\hline
$585$ & $\left<w_5,w_{13}\right>$ & $4$ & $27$ & $585$ & $\left<w_9,w_{13}\right>$ & $4$ & $26$\\
\hline
$585$ & $\left<w_9,w_{65}\right>$ & $4$ & $30$ & $585$ & $\left<w_{45},w_{65}\right>$ & $4$ & $27$\\
\hline
$595$ & $\left<w_7,w_{17}\right>$ & $4$ & $24$ & $595$ & $\left<w_{35},w_{85}\right>$ & $4$ & $21$\\
\hline
$595$ & $\left<w_5,w_{119}\right>$ & $4$ & $25$ & $598$ & $\left<w_{13},w_{23}\right>$ & $9$ & $44$\\
\hline
$598$ & $\left<w_{26},w_{46}\right>$ & $9$ & $48$ & $598$ & $\left<w_2,w_{299}\right>$ & $9$ & $54$\\
\hline
$627$ & $\left<w_{11},w_{19}\right>$ & $4$ & $27$ & $627$ & $\left<w_{33},w_{57}\right>$ & $4$ & $29$\\
\hline
$627$ & $\left<w_3,w_{209}\right>$ & $4$ & $28$ & $645$ & $\left<w_{15},w_{129}\right>$ & $2$ & $27$\\
\hline
$645$ & $\left<w_3,w_{215}\right>$ & $4$ & $30$ & $658$ & $\left<w_7,w_{47}\right>$ & $9$ & $52$\\
\hline
$663$ & $\left<w_3,w_{221}\right>$ & $4$ & $27$ & $670$ & $\left<w_{10},w_{134}\right>$ & $9$ & $54$\\
\hline
$670$ & $\left<w_2,w_{335}\right>$ & $9$ & $57$ & $770$ & $\left<w_2,w_7,w_{55}\right>$ & $9$ & $43$\\
\hline
$770$ & $\left<w_5,w_{11},w_{55}\right>$ & $9$ & $44$ & $770$ & $\left<w_2,w_{35},w_{55}\right>$ & $9$ & $44$\\
\hline
$770$ & $\left<w_{10},w_{14},w_{22}\right>$ & $9$ & $43$ & $910$ & $\left<w_5,w_7,w_{13}\right>$ & $9$ & $49$\\
\hline
$910$ & $\left<w_2,w_5,w_{91}\right>$ & $9$ & $52$ & $910$ & $\left<w_2,w_{13},w_{35}\right>$ & $9$ & $48$\\
\hline
$910$ & $\left<w_5,w_{14},w_{26}\right>$ & $9$ & $47$ & $910$ & $\left<w_{10},w_{13},w_{14}\right>$ & $9$ & $45$\\
    \hline
\end{longtable}
\end{center}
\end{prop}

\begin{proof}
    Using \texttt{Magma}, we calculate the number of $\F_q$ points on $X_0(N)/W_N$. It is now easy to check that $\#(X_0(N)/W_N)(\F_q)>4(q+1)$ in all these cases.

    For some curves we did this without the model of the curve using the function FpnpointsforQuotientcurveX0NWN() from \begin{center}
        \url{https://github.com/FrancescBars/Magma-functions-on-Quotient-Modular-Curves/blob/main/funcions.m}.
    \end{center} It computes the newforms $f$ such that the corresponding abelian variety $A_f$ is (up to isogeny) in the Jacobian decomposition of the curve. After that, it finds the characteristic polynomial \[P(x)=\prod_f (x^2-a_p(f)x+p)=\prod_1^{2g}(x-\alpha_i).\] The number of $\F_{p^k}$-points on the curve is then equal to $p^k+1-\sum_1^{2g}\alpha_i^k$.
    
    This function can fail to give the correct decomposition and hence return a higher number of points than the correct one. In these cases we found the number of $\F_q$-points via the model of the curve. More details regarding the code are on Github.

    The computations without the model were very fast, up to $1$ minute. Some computations that needed the model took longer. For example, the curve $X_0(560)/\left<w_5,w_{16}\right>$ is of genus $20$ and it took $2$ hours to count the $\F_9$-points on it.
\end{proof}

\section{Betti numbers}\label{betti_section}

Graded Betti numbers $\beta_{i,j}$ are helpful when determining the $\C$-gonality of a curve. We will use the indexation of Betti numbers as in \cite[Section 1.]{JeonPark05}. The results we mention can be found there and in \cite[Section 3.1.]{NajmanOrlic22}.

\begin{definition}
    For a curve $X$ and divisor $D$ of degree $d$, $g_d^r$ is a subspace $V$ of the Riemann-Roch space $L(D)$ such that $\dim V=r+1$.
\end{definition}

We see that the existence of $g_4^1$ is related to the existence of degree $4$ maps to $\PP^1$. Green's conjecture relates values of graded Betti numbers $\beta_{i,j}$ with the existence of $g_d^r$.

\begin{conj}[Green, \cite{Green84}]
    Let $X$ be a curve of genus $g$. Then $\beta_{p,2}\neq 0$ if and only if there exists a divisor $D$ on $X$ of degree $d$ such that a subspace $g_d^r$ of $L(D)$ satisfies $d\leq g-1$, $r=\ell(D)-1\geq1$, and $d-2r\leq p$.
\end{conj}

The "if" part of this conjecture has been proven in the same paper.

\begin{thm}[Green and Lazarsfeld, Appendix to \cite{Green84}]\label{thmGreenLazarsfeld}
    Let $X$ be a curve of genus $g$. If $\beta_{p,2}=0$, then there does not exist a divisor $D$ on $X$ of degree $d$ such that a subspace $g_d^r$ of $L(D)$ satisfies $d\leq g-1$, $r\geq1$, and $d-2r\leq p$.
\end{thm}

\begin{cor}[{\cite[Corollary 3.20]{Orlic2023}}]\label{betti0cor}
    Let $X$ be a curve of genus $g\geq5$ with $\beta_{2,2}=0$. Suppose that $X$ is neither hyperelliptic nor trigonal. Then $\textup{gon}_\C(X)\geq5$.
\end{cor}

Even though the "only if" part is still an open problem, we can say something more about $g_4^1$ from the value of $\beta_{2,2}$.

\begin{thm}[{\cite[Theorem 4.1, 4.4]{Schreyer91}}]\label{schreyer}
    Let $C\subseteq\PP^{g-1}$ be a reduced irreducible canonical curve over $\C$ of genus $g\geq7$. Then $\beta_{2,2}$ has one of the values in the following table.

    \begin{center}
        \begin{tabular}{|c|c|c|c|c|}
            \hline
             $\beta_{2,2}$ & $(g-4)(g-2)$ & $\displaystyle\binom{g-2}{2}-1$ & $g-4$ & $0$\\
             \hline
             \textup{linear series} & $\exists g_3^1$ & $\exists g_6^2\textup{ or } g_8^3$ & $\exists \textup{ a single } g_4^1$ & $\textup{no } g_4^1$\\
             \hline
        \end{tabular}
    \end{center}
\end{thm}

Notice that the indexation of Betti numbers is different here and in \cite{Schreyer91}, where they are indexed as $\beta_{2,4}$.

The existence of a single $g_4^1$ is useful as the following result tells us that in that case we have a rational morphism of degree $\leq4$ to $\PP^1$.

\begin{thm}[{\cite[Theorem 5]{RoeXarles2014}}]
    Let $C/k$ be a smooth projective curve. Suppose that for a fixed $r$ and $d$ there is only one $g_d^r$, giving a morphism $f:C_{k^{\textup{sep}}}\to \PP_{k^{\textup{sep}}}^r$. Then there exists a Brauer-Severi variety $\mathcal{P}$ defined over $k$ together with a $k$-morphism $g:C\to\mathcal{P}$ such that $g\oplus_k k^{\textup{sep}}:C_{k^{\textup{sep}}}\to\mathcal{P}_{k^{\textup{sep}}}\cong\PP_{k^{\textup{sep}}}^r$ is equal to $f$.
\end{thm}

\begin{cor}[{\cite[Corollary 5.7]{Orlic2025star}}]\label{unique_g14_cor}
    Let $C/\Q$ be smooth projective curve such that there is only one $g_d^1$. Then there exists a conic $\mathcal{P}$ defined over $\Q$ and a degree $d$ rational map $g:C\to\mathcal{P}$. If additionally $C$ has at least one rational point, then $\mathcal{P}$ is $\Q$-isomorphic to $\PP^1$.
\end{cor}

\begin{prop}\label{betti0prop}
    The $\C$-gonality of the quotient curve $X_0(N)/W_N$ is at least $5$ for the following pairs $(N,W_N)$. Here $g$ is the genus of $X_0(N)/W_N$.

    \begin{center}
\begin{longtable}{|c|c|c||c|c|c||c|c|c||c|c|c|}

\hline
\addtocounter{table}{-1}
$N$ & $W_N$ & $g$ & $N$ & $W_N$ & $g$ & $N$ & $W_N$ & $g$ & $N$ & $W_N$ & $g$\\
    \hline
$210$ & $\left<w_2,w_7\right>$ & $9$ & $210$ & $\left<w_5,w_7\right>$ & $7$ & $210$ & $\left<w_2,w_{35}\right>$ & $8$ & $210$ & $\left<w_3,w_{14}\right>$ & $9$\\
\hline
$210$ & $\left<w_3,w_{35}\right>$ & $7$ & $210$ & $\left<w_5,w_6\right>$ & $9$ & $210$ & $\left<w_5,w_{14}\right>$ & $8$ & $210$ & $\left<w_5,w_{21}\right>$ & $8$\\
\hline
$210$ & $\left<w_6,w_{70}\right>$ & $9$ & $210$ & $\left<w_{14},w_{30}\right>$ & $8$ & $210$ & $\left<w_2,w_{105}\right>$ & $9$ & $210$ & $\left<w_5,w_{42}\right>$ & $9$\\
\hline
$228$ & $\left<w_3,w_{19}\right>$ & $9$ & $228$ & $\left<w_{12},w_{19}\right>$ & $8$ & $228$ & $\left<w_{12},w_{57}\right>$ & $8$ & $234$ & $\left<w_2,w_9\right>$ & $9$\\
\hline
$234$ & $\left<w_2,w_{13}\right>$ & $8$ & $234$ & $\left<w_9,w_{13}\right>$ & $8$ & $240$ & $\left<w_3,w_5\right>$ & $8$ & $240$ & $\left<w_5,w_48\right>$ & $9$\\
\hline
$246$ & $\left<w_2,w_{41}\right>$ & $8$ & $246$ & $\left<w_3,w_{41}\right>$ & $7$ & $246$ & $\left<w_3,w_{82}\right>$ & $9$ & $246$ & $\left<w_6,w_{82}\right>$ & $9$\\
\hline
$246$ & $\left<w_6,w_{41}\right>$ & $7$ & $246$ & $\left<w_2,w_{123}\right>$ & $7$ & $258$ & $\left<w_3,w_{43}\right>$ & $9$ & $258$ & $\left<w_3,w_{86}\right>$ & $7$\\
\hline
$258$ & $\left<w_2,w_{129}\right>$ & $8$ & $260$ & $\left<w_{13},w_{20}\right>$ & $8$ & $260$ & $\left<w_5,w_{52}\right>$ & $8$ & $264$ & $\left<w_3,w_8\right>$ & $9$\\
\hline
$264$ & $\left<w_3,w_{88}\right>$ & $9$ & $270$ & $\left<w_5,w_{27}\right>$ & $8$ & $270$ & $\left<w_5,w_{54}\right>$ & $8$ & $270$ & $\left<w_2,w_{135}\right>$ & $7$\\
\hline
$273$ & $\left<w_7,w_{13}\right>$ & $9$ & $273$ & $\left<w_7,w_{39}\right>$ & $8$ & $273$ & $\left<w_{21},w_{39}\right>$ & $9$ & $273$ & $\left<w_{13},w_{21}\right>$ & $8$\\
\hline
$280$ & $\left<w_7,w_8\right>$ & $9$ & $280$ & $\left<w_7,w_{40}\right>$ & $9$ & $290$ & $\left<w_2,w_{29}\right>$ & $9$ & $290$ & $\left<w_5,w_{29}\right>$ & $8$\\
\hline
$290$ & $\left<w_{10},w_{29}\right>$ & $7$ & $290$ & $\left<w_5,w_{58}\right>$ & $8$ & $290$ & $\left<w_2,w_{145}\right>$ & $7$ & $294$ & $\left<w_2,w_{49}\right>$ & $9$\\
\hline
$294$ & $\left<w_3,w_{49}\right>$ & $9$ & $294$ & $\left<w_6,w_{49}\right>$ & $9$ & $294$ & $\left<w_3,w_{98}\right>$ & $7$ & $294$ & $\left<w_2,w_{147}\right>$ & $8$\\
\hline
$306$ & $\left<w_{17},w_{18}\right>$ & $9$ & $308$ & $\left<w_4,w_7\right>$ & $9$ & $308$ & $\left<w_7,w_{44}\right>$ & $8$ & $308$ & $\left<w_4,w_{77}\right>$ & $9$\\
\hline
$308$ & $\left<w_{11},w_{28}\right>$ & $9$ & $310$ & $\left<w_2,w_{31}\right>$ & $9$ & $310$ & $\left<w_{10},w_{31}\right>$ & $8$ & $310$ & $\left<w_{10},w_{62}\right>$ & $9$\\
\hline
$310$ & $\left<w_2,w_{155}\right>$ & $8$ & $312$ & $\left<w_3,w_{13}\right>$ & $9$ & $312$ & $\left<w_8,w_{39}\right>$ & $8$ & $312$ & $\left<w_3,w_{104}\right>$ & $9$\\
\hline
$318$ & $\left<w_3,w_{53}\right>$ & $7$ & $318$ & $\left<w_6,w_{106}\right>$ & $8$ & $318$ & $\left<w_2,w_{159}\right>$ & $7$ & $322$ & $\left<w_7,w_{23}\right>$ & $9$\\
\hline
$322$ & $\left<w_{14},w_{46}\right>$ & $9$ & $322$ & $\left<w_2,w_{161}\right>$ & $9$ & $342$ & $\left<w_{18},w_{38}\right>$ & $9$ & $345$ & $\left<w_{15},w_{69}\right>$ & $9$\\
\hline
$345$ & $\left<w_5,w_{69}\right>$ & $8$ & $348$ & $\left<w_{12},w_{87}\right>$ & $7$ & $350$ & $\left<w_{14},w_{50}\right>$ & $9$ & $350$ & $\left<w_2,w_{175}\right>$ & $8$\\
\hline
$354$ & $\left<w_6,w_{59}\right>$ & $8$ & $366$ & $\left<w_6,w_{122}\right>$ & $9$ & $370$ & $\left<w_{10},w_{74}\right>$ & $9$ & $374$ & $\left<w_{17},w_{22}\right>$ & $9$\\
\hline$374$ & $\left<w_2,w_{187}\right>$ & $8$ & $385$ & $\left<w_{11},w_{35}\right>$ & $9$ & $385$ & $\left<w_{35},w_{55}\right>$ & $8$ & $385$ & $\left<w_7,w_{55}\right>$ & $9$\\
\hline
\caption*{}
\end{longtable}
\end{center}
\end{prop}

\begin{proof}
    For all these curves we compute that $\beta_{2,2}=0$. We know from \cite{FurumotoHasegawa1999,HasegawaShimura2006} that these curves are neither hyperelliptic nor trigonal. Thus by Corollary \ref{betti0cor} we conclude that they are not tetragonal over $\C$.

    The computations for $g=7$ curves took only several seconds, for $g=8$ up to several minutes, and for $g=9$ they could take several hours. The computation for the genus $9$ curve $X_0(234)/\left<w_2,w_9\right>$ took around $4$ hours. 
    
    The Betti numbers were computed with the \texttt{Magma} function BettiNumber(\_,2,4). The indexation of $\beta_{i,j}$ in that function is different from that in this paper. It is the same indexation as in \cite{Schreyer1986} and \cite{Schreyer91}, and the reader can consult \cite[Table 1]{Schreyer1986} to check this result.
\end{proof}

\begin{prop}\label{betti_positive_prop}
    The quotient curve $X_0(N)/W_N$ has $\Q$ and $\C$-gonality equal to $4$ for the following pairs $(N,W_N)$:

    \begin{center}
\begin{longtable}{|c|c|c|c||c|c|c|c||c|c|c|c|}

\hline
\addtocounter{table}{-1}
$N$ & $W_N$ & $g$ & $\beta_{2,2}$ & $N$ & $W_N$ & $g$ & $\beta_{2,2}$ & $N$ & $W_N$ & $g$ & $\beta_{2,2}$\\
    \hline
$228$ & $\left<w_3,w_{76}\right>$ & $8$ & $4$ & $240$ & $\left<w_5,w_{16}\right>$ & $8$ & $4$ & $240$ & $\left<w_{15},w_{16}\right>$ & $7$ & $3$\\
\hline
$258$ & $\left<w_2,w_{43}\right>$ & $8$ & $4$ & $258$ & $\left<w_6,w_{86}\right>$ & $7$ & $3$ & $264$ & $\left<w_8,w_{33}\right>$ & $8$ & $4$\\
\hline
$264$ & $\left<w_{11},w_{24}\right>$ & $8$ & $4$ & $273$ & $\left<w_3,w_7\right>$ & $8$ & $4$ & $273$ & $\left<w_3,w_{13}\right>$ & $8$ & $4$\\
\hline
$280$ & $\left<w_{35},w_{56}\right>$ & $8$ & $4$ & $280$ & $\left<w_5,w_{56}\right>$ & $8$ & $4$ & & & &\\
\hline
\caption*{}
\end{longtable}
\end{center}
\end{prop}

\begin{proof}
    For these curves we computed that $\beta_{2,2}=g-4$. Thus there exists a unique $g_4^1$ by Theorem \ref{schreyer}. Since all quotients of $X_0(N)$ have a rational cusp, Corollary \ref{unique_g14_cor} now tells us that we have a rational map to $\PP^1$ of degree $4$. Therefore, all these curves have $\Q$ and $\C$-gonality $4$ as they are not hyperelliptic nor trigonal by \cite{FurumotoHasegawa1999, HasegawaShimura2006}.
\end{proof}

\section{Castelnuovo-Severi inequality}\label{CSsection}

This is a very useful tool for producing a lower bound on the $\C$-gonality (see \cite[Theorem 3.11.3]{Stichtenoth09} for a proof).
\begin{prop}[Castelnuovo-Severi inequality] 
\label{tm:CS}
Let $k$ be a perfect field, and let $X,\ Y, \ Z$ be curves over $k$ with genera $g(X), g(Y), g(Z)$. Let non-constant morphisms $\pi_Y:X\rightarrow Y$ and $\pi_Z:X\rightarrow Z$ over $k$ be given, and let their degrees be $m$ and $n$, respectively. Assume that there is no morphism $X\rightarrow X'$ of degree $>1$ through which both $\pi_Y$ and $\pi_Z$ factor. Then the following inequality holds:
\begin{equation} \label{eq:CS}
g(X)\leq m \cdot g(Y)+n\cdot g(Z) +(m-1)(n-1).
\end{equation}
\end{prop}
As we can see, Castelnuovo-Severi inequality is useful when we have a low degree morphism from a curve $X$ of high genus to a curve $Y$ of low genus. We will use this in the following results.

\begin{prop}\label{CSprop_star}
The $\C$-gonality of the quotient curve $X_0(N)/W_N$ is at least $5$ for the following pairs $(N,W_N)$. Here $g$ denotes the genus of $X_0(N)/W_N$ and $g^*$ denotes the genus of the curve $X_0^*(N)$.

    \begin{center}
\begin{longtable}{|c|c|c|c||c|c|c|c||c|c|c|c|}

\hline
\addtocounter{table}{-1}
$N$ & $W_N$ & $g$ & $g^*$ & $N$ & $W_N$ & $g$ & $g^*$ & $N$ & $W_N$ & $g$ & $g^*$\\
    \hline
$240$ & $\left<w_3,w_{16}\right>$ & $10$ & $3$ & $246$ & $\left<w_2,w_3\right>$ & $10$ & $3$ & $258$ & $\left<w_2,w_3\right>$ & $10$ & $3$\\
\hline
$258$ & $\left<w_6,w_{43}\right>$ & $10$ & $3$ & $270$ & $\left<w_2,w_5\right>$ & $11$ & $3$ & $270$ & $\left<w_2,w_{27}\right>$ & $10$ & $3$\\
\hline
$270$ & $\left<w_{10},w_{27}\right>$ & $10$ & $3$ & $282$ & $\left<w_2,w_3\right>$ & $12$ & $3$ & $282$ & $\left<w_6,w_{94}\right>$ & $11$ & $3$\\
\hline$282$ & $\left<w_3,w_{94}\right>$ & $11$ & $3$ & $282$ & $\left<w_2,w_{141}\right>$ & $10$ & $3$ & $290$ & $\left<w_2,w_5\right>$ & $10$ & $3$\\
\hline
$290$ & $\left<w_{10},w_{58}\right>$ & $10$ & $3$ & $294$ & $\left<w_2,w_3\right>$ & $10$ & $3$ & $310$ & $\left<w_2,w_5\right>$ & $12$ & $3$\\
\hline
$310$ & $\left<w_5,w_{62}\right>$ & $11$ & $3$ & $312$ & $\left<w_3,w_8\right>$ & $13$ & $3$ & $312$ & $\left<w_8,w_{13}\right>$ & $10$ & $3$\\
\hline
$312$ & $\left<w_{13},w_{23}\right>$ & $12$ & $3$ & $318$ & $\left<w_2,w_3\right>$ & $13$ & $3$ & $318$ & $\left<w_2,w_{53}\right>$ & $12$ & $3$\\
\hline
$318$ & $\left<w_3,w_{106}\right>$ & $12$ & $3$ & $318$ & $\left<w_6,w_{53}\right>$ & $10$ & $3$ & $322$ & $\left<w_2,w_{23}\right>$ & $12$ & $4$\\
\hline
$342$ & $\left<w_2,w_9\right>$ & $13$ & $4$ & $342$ & $\left<w_2,w_{19}\right>$ & $12$ & $4$ & $348$ & $\left<w_3,w_{29}\right>$ & $10$ & $3$\\
\hline
$348$ & $\left<w_{12},w_{29}\right>$ & $13$ & $3$ & $348$ & $\left<w_3,w_{116}\right>$ & $10$ & $3$ & $350$ & $\left<w_2,w_7\right>$ & $12$ & $4$\\
\hline
$350$ & $\left<w_2,w_{25}\right>$ & $13$ & $4$ & $354$ & $\left<w_2,w_3\right>$ & $14$ & $3$ & $354$ & $\left<w_6,w_{118}\right>$ & $14$ & $3$\\
\hline
$354$ & $\left<w_3,w_{118}\right>$ & $13$ & $3$ & $354$ & $\left<w_2,w_{177}\right>$ & $12$ & $3$ & $366$ & $\left<w_2,w_3\right>$ & $15$ & $4$\\
\hline
$366$ & $\left<w_2,w_{61}\right>$ & $13$ & $4$ & $366$ & $\left<w_6,w_{61}\right>$ & $14$ & $4$ & $370$ & $\left<w_2,w_5\right>$ & $13$ & $4$\\
\hline
$370$ & $\left<w_5,w_{37}\right>$ & $12$ & $4$ & $370$ & $\left<w_{10},w_{37}\right>$ & $12$ & $4$ & $374$ & $\left<w_2,w_{11}\right>$ & $13$ & $4$\\
\hline
$374$ & $\left<w_2,w_{17}\right>$ & $12$ & $4$ & $374$ & $\left<w_{22},w_{34}\right>$ & $12$ & $4$ & $378$ & $\left<w_2,w_7\right>$ & $15$ & $5$\\
\hline
$385$ & $\left<w_7,w_{11}\right>$ & $12$ & $4$ & $396$ & $\left<w_4,w_9\right>$ & $15$ & $5$ & $402$ & $\left<w_2,w_3\right>$ & $16$ & $5$\\
\hline
$402$ & $\left<w_3,w_{67}\right>$ & $15$ & $5$ & $402$ & $\left<w_6,w_{67}\right>$ & $15$ & $5$ & $406$ & $\left<w_2,w_7\right>$ & $14$ & $5$\\
\hline
$406$ & $\left<w_2,w_{29}\right>$ & $15$ & $5$ & $410$ & $\left<w_2,w_5\right>$ & $14$ & $5$ & $410$ & $\left<w_{10},w_{82}\right>$ & $14$ & $5$\\
\hline
$414$ & $\left<w_2,w_9\right>$ & $17$ & $5$ & $414$ & $\left<w_2,w_{23}\right>$ & $14$ & $5$ & $414$ & $\left<w_9,w_{46}\right>$ & $15$ & $5$\\
\hline
$418$ & $\left<w_2,w_{11}\right>$ & $14$ & $5$ & $418$ & $\left<w_{11},w_{38}\right>$ & $14$ & $5$ & $435$ & $\left<w_3,w_5\right>$ & $14$ & $5$\\
\hline
$435$ & $\left<w_{15},w_{87}\right>$ & $15$ & $5$ & $438$ & $\left<w_2,w_3\right>$ & $17$ & $5$ & $438$ & $\left<w_2,w_{73}\right>$ & $14$ & $5$\\
\hline
$438$ & $\left<w_3,w_{73}\right>$ & $15$ & $5$ & $438$ & $\left<w_6,w_{73}\right>$ & $17$ & $5$ & $438$ & $\left<w_2,w_{219}\right>$ & $14$ & $5$\\
\hline
$440$ & $\left<w_5,w_8\right>$ & $16$ & $5$ & $440$ & $\left<w_8,w_{11}\right>$ & $17$ & $5$ & $440$ & $\left<w_5,w_{88}\right>$ & $14$ & $5$\\
\hline
$442$ & $\left<w_2,w_{17}\right>$ & $14$ & $5$ & $442$ & $\left<w_{13},w_{34}\right>$ & $14$ & $5$ & $444$ & $\left<w_{12},w_{37}\right>$ & $16$ & $5$\\
\hline
$444$ & $\left<w_3,w_{148}\right>$ & $16$ & $5$ & $456$ & $\left<w_3,w_8\right>$ & $18$ & $7$ & $456$ & $\left<w_3,w_{19}\right>$ & $19$ & $7$\\
\hline
$465$ & $\left<w_3,w_5\right>$ & $15$ & $5$ & $465$ & $\left<w_3,w_{31}\right>$ & $16$ & $5$ & $465$ & $\left<w_5,w_{93}\right>$ & $14$ & $5$\\
\hline
$470$ & $\left<w_2,w_5\right>$ & $17$ & $6$ & $470$ & $\left<w_2,w_{47}\right>$ & $16$ & $6$ & $470$ & $\left<w_5,w_{47}\right>$ & $16$ & $6$\\
\hline
$494$ & $\left<w_2,w_{13}\right>$ & $17$ & $5$ & $494$ & $\left<w_2,w_{19}\right>$ & $16$ & $5$ & $494$ & $\left<w_{13},w_{19}\right>$ & $14$ & $5$\\
\hline
$494$ & $\left<w_{19},w_{26}\right>$ & $14$ & $5$ & $495$ & $\left<w_5,w_9\right>$ & $17$ & $5$ & $495$ & $\left<w_5,w_{11}\right>$ & $15$ & $5$\\
\hline
$504$ & $\left<w_7,w_8\right>$ & $19$ & $7$ & $504$ & $\left<w_8,w_9\right>$ & $21$ & $7$ & $520$ & $\left<w_5,w_{13}\right>$ & $20$ & $8$\\
\hline
$528$ & $\left<w_3,w_{11}\right>$ & $22$ & $9$ & $528$ & $\left<w_3,w_{16}\right>$ & $22$ & $9$ & $555$ & $\left<w_3,w_5\right>$ & $19$ & $5$\\
\hline
$555$ & $\left<w_5,w_{37}\right>$ & $15$ & $5$ & $555$ & $\left<w_3,w_{37}\right>$ & $15$ & $5$ & $555$ & $\left<w_{15},w_{37}\right>$ & $17$ & $5$\\
\hline
$558$ & $\left<w_2,w_9\right>$ & $23$ & $7$ & $558$ & $\left<w_2,w_{13}\right>$ & $21$ & $7$ & $558$ & $\left<w_9,w_{62}\right>$ & $19$ & $7$\\
\hline
$560$ & $\left<w_7,w_{16}\right>$ & $22$ & $9$ & $560$ & $\left<w_5,w_7\right>$ & $22$ & $9$ & $564$ & $\left<w_3,w_{47}\right>$ & $16$ & $6$\\
\hline
$564$ & $\left<w_{12},w_{141}\right>$ & $16$ & $6$ & $564$ & $\left<w_3,w_{188}\right>$ & $19$ & $6$ & $572$ & $\left<w_{11},w_{52}\right>$ & $17$ & $5$\\
\hline
$572$ & $\left<w_{13},w_{44}\right>$ & $18$ & $5$ & $574$ & $\left<w_2,w_7\right>$ & $21$ & $5$ & $574$ & $\left<w_2,w_{41}\right>$ & $18$ & $5$\\
\hline
$574$ & $\left<w_{14},w_{41}\right>$ & $17$ & $5$ & $574$ & $\left<w_7,w_{82}\right>$ & $18$ & $5$ & $595$ & $\left<w_5,w_7\right>$ & $16$ & $5$\\
\hline
$595$ & $\left<w_5,w_{17}\right>$ & $18$ & $5$ & $595$ & $\left<w_{17},w_{35}\right>$ & $14$ & $5$ &  $595$ & $\left<w_7,w_{85}\right>$ & $16$ & $5$\\
\hline
$598$ & $\left<w_2,w_{13}\right>$ & $21$ & $6$ & $598$ & $\left<w_2,w_{23}\right>$ & $18$ & $6$ & $598$ & $\left<w_{23},w_{26}\right>$ & $17$ & $6$\\
\hline
$598$ & $\left<w_{13},w_{46}\right>$ & $20$ & $6$ & $620$ & $\left<w_5,w_{31}\right>$ & $16$ & $6$ & $620$ & $\left<w_{20},w_{31}\right>$ & $16$ & $6$\\
\hline
$620$ & $\left<w_{20},w_{124}\right>$ & $19$ & $6$ & $620$ & $\left<w_5,w_{124}\right>$ & $19$ & $6$ & $627$ & $\left<w_3,w_{11}\right>$ & $19$ & $6$\\
\hline
$627$ & $\left<w_3,w_{19}\right>$ & $19$ & $6$ & $627$ & $\left<w_{19},w_{33}\right>$ & $17$ & $6$ & $627$ & $\left<w_{11},w_{57}\right>$ & $17$ & $6$\\
\hline
$630$ & $\left<w_2,w_5,w_9\right>$ & $15$ & $5$ & $630$ & $\left<w_2,w_5,w_7\right>$ & $14$ & $5$ & $630$ & $\left<w_2,w_7,w_9\right>$ & $15$ & $5$\\
\hline
$630$ & $\left<w_7,w_9,w_{10}\right>$ & $15$ & $5$ & $645$ & $\left<w_3,w_5\right>$ & $21$ & $5$ & $645$ & $\left<w_3,w_{43}\right>$ & $19$ & $5$\\
\hline
$645$ & $\left<w_5,w_{43}\right>$ & $14$ & $5$ & $645$ & $\left<w_{15},w_{43}\right>$ & $20$ & $5$ & $645$ & $\left<w_5,w_{129}\right>$ & $16$ & $5$\\
\hline
$658$ & $\left<w_2,w_7\right>$ & $24$ & $7$ & $658$ & $\left<w_{14},w_{47}\right>$ & $18$ & $7$ & $658$ & $\left<w_{14},w_{94}\right>$ & $19$ & $7$\\
\hline
$658$ & $\left<w_7,w_{94}\right>$ & $21$ & $7$ & $658$ & $\left<w_2,w_{329}\right>$ & $20$ & $7$ & $663$ & $\left<w_3,w_{13}\right>$ & $21$ & $5$\\
\hline
$663$ & $\left<w_3,w_{17}\right>$ & $17$ & $5$ & $663$ & $\left<w_{13},w_{17}\right>$ & $15$ & $5$ & $663$ & $\left<w_{13},w_{51}\right>$ & $15$ & $5$\\
\hline
$663$ & $\left<w_{39},w_{51}\right>$ & $15$ & $5$ & $663$ & $\left<w_{17},w_{39}\right>$ & $15$ & $5$ & $670$ & $\left<w_2,w_5\right>$ & $25$ & $6$\\
\hline
$670$ & $\left<w_2,w_{67}\right>$ & $22$ & $6$ & $670$ & $\left<w_5,w_{67}\right>$ & $16$ & $6$ & $670$ & $\left<w_{10},w_{67}\right>$ & $24$ & $6$\\
\hline
$670$ & $\left<w_5,w_{134}\right>$ & $20$ & $6$ & $714$ & $\left<w_2,w_3,w_7\right>$ & $17$ & $5$ & $714$ & $\left<w_2,w_3,w_{17}\right>$ & $17$ & $5$\\
\hline
$714$ & $\left<w_2,w_7,w_{51}\right>$ & $16$ & $5$ & $714$ & $\left<w_2,w_{17},w_{21}\right>$ & $14$ & $5$ & $714$ & $\left<w_3,w_7,w_{34}\right>$ & $16$ & $5$\\
\hline
$714$ & $\left<w_3,w_{14},w_{17}\right>$ & $15$ & $5$ & $714$ & $\left<w_6,w_7,w_{34}\right>$ & $17$ & $5$ & $714$ & $\left<w_6,w_{14},w_{17}\right>$ & $16$ & $5$\\
\hline
$770$ & $\left<w_2,w_5,w_7\right>$ & $16$ & $5$ & $770$ & $\left<w_2,w_5,w_{11}\right>$ & $14$ & $5$ & $770$ & $\left<w_2,w_7,w_{11}\right>$ & $17$ & $5$\\
\hline
$770$ & $\left<w_5,w_7,w_{11}\right>$ & $14$ & $5$ & $770$ & $\left<w_2,w_5,w_{77}\right>$ & $14$ & $5$ & $770$ & $\left<w_2,w_{11},w_{35}\right>$ & $14$ & $5$\\
\hline
$770$ & $\left<w_7,w_{10},w_{11}\right>$ & $14$ & $5$ & $770$ & $\left<w_5,w_{14},w_{22}\right>$ & $16$ & $5$ & $770$ & $\left<w_7,w_{10},w_{22}\right>$ & $14$ & $5$\\
\hline
$780$ & $\left<w_3,w_5,w_{13}\right>$ & $16$ & $6$ & $780$ & $\left<w_3,w_5,w_{52}\right>$ & $16$ & $6$ & $780$ & $\left<w_5,w_{12},w_{13}\right>$ & $18$ & $6$\\
\hline
$780$ & $\left<w_3,w_{20},w_{52}\right>$ & $18$ & $6$ & $780$ & $\left<w_{12},w_{13},w_{15}\right>$ & $16$ & $6$ & $780$ & $\left<w_{12},w_{15},w_{39}\right>$ & $16$ & $6$\\
\hline
$798$ & $\left<w_2,w_3,w_7\right>$ & $18$ & $5$ & $798$ & $\left<w_2,w_3,w_{19}\right>$ & $18$ & $5$ & $798$ & $\left<w_2,w_7,w_{19}\right>$ & $15$ & $5$\\
\hline
$798$ & $\left<w_3,w_7,w_{19}\right>$ & $15$ & $5$ & $798$ & $\left<w_2,w_7,w_{57}\right>$ & $14$ & $5$ & $798$ & $\left<w_3,w_7,w_{38}\right>$ & $14$ & $5$\\
\hline
$798$ & $\left<w_3,w_{14},w_{19}\right>$ & $15$ & $5$ & $798$ & $\left<w_6,w_7,w_{19}\right>$ & $15$ & $5$ & $798$ & $\left<w_2,w_{21},w_{57}\right>$ & $17$ & $5$\\
\hline
$798$ & $\left<w_6,w_{14},w_{38}\right>$ & $16$ & $5$ & $910$ & $\left<w_2,w_5,w_7\right>$ & $18$ & $5$ & $910$ & $\left<w_2,w_5,w_{13}\right>$ & $19$ & $5$\\
\hline
$910$ & $\left<w_2,w_7,w_{13}\right>$ & $17$ & $5$ & $910$ & $\left<w_2,w_7,w_65\right>$ & $14$ & $5$ & $910$ & $\left<w_5,w_7,w_{26}\right>$ & $17$ & $5$\\
\hline
$910$ & $\left<w_5,w_{13},w_{14}\right>$ & $19$ & $5$ & $910$ & $\left<w_7,w_{10},w_{13}\right>$ & $18$ & $5$ & $910$ & $\left<w_2,w_{35},w_{65}\right>$ & $18$ & $5$\\
\hline
$910$ & $\left<w_7,w_{10},w_{26}\right>$ & $14$ & $5$ & $910$ & $\left<w_{10},w_{14},w_{26}\right>$ & $14$ & $5$ & $990$ & $\left<w_2,w_5,w_9\right>$ & $25$ & $8$\\
\hline
$990$ & $\left<w_2,w_5,w_{11}\right>$ & $23$ & $8$ & $990$ & $\left<w_2,w_9,w_{11}\right>$ & $23$ & $8$ & $990$ & $\left<w_5,w_9,w_{22}\right>$ & $23$ & $8$\\
\hline
$990$ & $\left<w_5,w_{11},w_{18}\right>$ & $22$ & $8$ & $990$ & $\left<w_9,w_{10},w_{22}\right>$ & $21$ & $8$ & & & &\\

\hline
\caption*{}
\end{longtable}
\end{center}
\end{prop}

\begin{proof}
    These curves $X_\Delta(N)$ are neither hyperelliptic nor trigonal by \cite{FurumotoHasegawa1999, HasegawaShimura2006}. Therefore, their $\C$-gonality is at least $4$.

    Suppose that there is a degree $4$ morphism from $X_0(N)/W_N$ to $\mathbb{P}^1$ for some pair $(N,W_N)$ from this table. We apply the Castelnuovo-Severi inequality with the degree $2$ quotient map $X_0(N)/W_N\to X_0^*(N)$ and this hypothetical degree $4$ morphism. We can easily check that for all these curves we have \[g>2\cdot g^*+2\cdot0+1\cdot3=2\cdot g^*+3.\]
    Therefore, this degree $4$ morphism to $\PP^1$ would have to factor through $X_0^*(N)$, implying that $X_0^*(N)$ would have to be hyperelliptic. However, \cite{Hasegawa1997} tells us that none of these curves $X_0^*(N)$ are hyperelliptic and we get a contradiction.
\end{proof}

\begin{remark}
    Proposition \ref{CSprop_star} fixes a small typo in \cite{HasegawaShimura2006}. In that paper it is stated that the curve $X_0(270)/\left<w_{10},w_{27}\right>$ is trigonal of genus $7$. However, this curve is of genus $10$ and is not trigonal as we have just seen (apply the Castelnuovo-Severi inequality with the degree $2$ map to $X_0^*(270)$ and the hypothetical degree $3$ map to $\PP^1$). 
    
    Instead, the curve $X_0(270)/\left<w_{10},w_{54}\right>$ is trigonal of genus $7$. We can prove its trigonality by computing the Betti number $\beta_{2,2}=20=(g-4)(g-2)$ which proves the existence of $g_3^1$ by Theorem \ref{schreyer}.
\end{remark}

\begin{prop}\label{CSprop_other}
The $\C$-gonality of the quotient curve $X_0(N)/W_N$ is at least $5$ for the following pairs $(N,W_N)$. Here $g$ denotes the genus of $X_0(N)/W_N$ and $g'$ denotes the genus of $X_0(N)/W_n'$.

    \begin{center}
\begin{longtable}{|c|c|c|c|c||c|c|c|c|c|}

\hline
\addtocounter{table}{-1}
$N$ & $W_N$ & $g$ & $W_N'$ & $g'$ & $N$ & $W_N$ & $g$ & $W_N'$ & $g'$\\
    \hline
$210$ & $\left<w_2,w_3\right>$ & $10$ & $\left<w_2,w_3,w_{35}\right>$ & $3$ & $210$ & $\left<w_2,w_5\right>$ & $10$ & $\left<w_2,w_5,w_7\right>$ & $3$\\
 & $\left<w_3,w_5\right>$ & $10$ & $\left<w_3,w_5,w_7\right>$ & $3$ & & $\left<w_3,w_7\right>$ & $10$ & $\left<w_3,w_5,w_7\right>$ & $3$\\
 & $\left<w_3,w_{1}\right>$ & $11$ & $\left<w_3,w_{10},w_{14}\right>$ & $3$ & & $\left<w_6,w_7\right>$ & $10$ & $\left<w_5,w_6,w_7\right>$ & $3$\\
 & $\left<w_7,w_{10}\right>$ & $11$ & $\left<w_2,w_5,w_7\right>$ & $3$ & & $\left<w_7,w_{15}\right>$ & $10$ & $\left<w_3,w_5,w_7\right>$ & $3$\\
 & $\left<w_{10},w_{42}\right>$ & $10$ & $\left<w_3,w_{10},w_{14}\right>$ & $3$ & & $\left<w_3,w_{70}\right>$ & $10$ & $\left<w_2,w_3,w_{35}\right>$ & $3$\\
 & $\left<w_7,w_{30}\right>$ & $10$ & $\left<w_5,w_6,w_7\right>$ & $3$ & & & & &\\ 
\hline
$330$ & $\left<w_2,w_3\right>$ & $16$ & $\left<w_2,w_3,w_{11}\right>$ & $6$ & $330$ & $\left<w_2,w_5\right>$ & $17$ & $\left<w_2,w_5,w_{11}\right>$ & $6$\\
 & $\left<w_3,w_5\right>$ & $17$ & $\left<w_3,w_5,w_{22}\right>$ & $6$ & & $\left<w_3,w_{11}\right>$ & $14$ & $\left<w_3,w_{10},w_{11}\right>$ & $5$\\
 & $\left<w_2,w_{15}\right>$ & $16$ & $\left<w_2,w_{11},w_{15}\right>$ & $6$ & & $\left<w_3,w_{10}\right>$ & $16$ & $\left<w_3,w_{10},w_{11}\right>$ & $6$\\
 & $\left<w_3,w_{55}\right>$ & $16$ & $\left<w_2,w_3,w_{55}\right>$ & $6$ & & $\left<w_5,w_6\right>$ & $15$ & $\left<w_5,w_6,w_{11}\right>$ & $5$\\
 & $\left<w_5,w_{33}\right>$ & $16$ & $\left<w_5,w_6,w_{22}\right>$ & $6$ & & $\left<w_{11},w_{15}\right>$ & $13$ & $\left<w_6,w_{10},w_{11}\right>$ & $4$\\
 & $\left<w_6,w_{10}\right>$ & $16$ & $\left<w_6,w_{10},w_{11}\right>$ & $4$ & & $\left<w_6,w_{22}\right>$ & $16$ & $\left<w_5,w_6,w_{22}\right>$ & $6$\\
 & $\left<w_{10},w_{22}\right>$ & $17$ & $\left<w_2,w_5,w_{11}\right>$ & $6$ & & $\left<w_{15},w_{33}\right>$ & $17$ & $\left<w_2,w_{15},w_{33}\right>$ & $6$\\
 & $\left<w_6,w_{55}\right>$ & $15$ & $\left<w_5,w_6,w_{11}\right>$ & $5$ & & $\left<w_{10},w_{33}\right>$ & $16$ & $\left<w_3,w_{10},w_{11}\right>$ & $5$\\
 & $\left<w_{15},w_{22}\right>$ & $16$ & $\left<w_2,w_{11},w_{15}\right>$ & $6$ & & $\left<w_6,w_{110}\right>$ & $12$ & $\left<w_6,w_{10},w_{11}\right>$ & $4$\\
 & $\left<w_{10},w_{66}\right>$ & $14$ & $\left<w_6,w_{10},w_{11}\right>$ & $4$ & & $\left<w_{15},w_{66}\right>$ & $12$ & $\left<w_6,w_{10},w_{11}\right>$ & $4$\\
 & $\left<w_{30},w_{55}\right>$ & $14$ & $\left<w_6,w_{10},w_{11}\right>$ & $4$ & & $\left<w_5,w_{66}\right>$ & $14$ & $\left<w_6,w_{10},w_{11}\right>$ & $4$\\
 & $\left<w_{11},w_{30}\right>$ & $12$ & $\left<w_6,w_{10},w_{11}\right>$ & $4$ & & & & &\\
\hline
$390$ & $\left<w_2,w_3\right>$ & $20$ & $\left<w_2,w_3,w_{13}\right>$ & $7$ & $390$ & $\left<w_2,w_5\right>$ & $20$ & $\left<w_2,w_5,w_{13}\right>$ & $8$\\
 & $\left<w_2,w_{13}\right>$ & $17$ & $\left<w_2,w_{13},w_{15}\right>$ & $6$ & & $\left<w_3,w_5\right>$ & $20$ & $\left<w_3,w_5,w_{13}\right>$ & $6$\\
 & $\left<w_3,w_{13}\right>$ & $16$ & $\left<w_3,w_5,w_{13}\right>$ & $6$ & & $\left<w_5,w_{13}\right>$ & $18$ & $\left<w_3,w_5,w_{13}\right>$ & $6$\\
 & $\left<w_2,w_{15}\right>$ & $19$ & $\left<w_2,w_{13},w_{15}\right>$ & $6$ & & $\left<w_2,w_{39}\right>$ & $16$ & $\left<w_2,w_5,w_{39}\right>$ & $6$\\
 & $\left<w_2,w_{65}\right>$ & $18$ & $\left<w_2,w_{15},w_{39}\right>$ & $7$ & & $\left<w_3,w_{10}\right>$ & $19$ & $\left<w_3,w_{10},w_{13}\right>$ & $7$\\
 & $\left<w_3,w_{26}\right>$ & $17$ & $\left<w_3,w_{10},w_{26}\right>$ & $6$ & & $\left<w_5,w_6\right>$ & $19$ & $\left<w_5,w_6,w_{26}\right>$ & $5$\\
 & $\left<w_5,w_{26}\right>$ & $17$ & $\left<w_5,w_6,w_{26}\right>$ & $5$ & & $\left<w_6,w_{13}\right>$ & $20$ & $\left<w_2,w_3,w_{13}\right>$ & $7$\\
 & $\left<w_{10},w_{13}\right>$ & $20$ & $\left<w_2,w_5,w_{13}\right>$ & $8$ & & $\left<w_{13},w_{15}\right>$ & $17$ & $\left<w_3,w_5,w_{13}\right>$ & $6$\\
 & $\left<w_6,w_{10}\right>$ & $20$ & $\left<w_6,w_{10},w_{26}\right>$ & $5$ & & $\left<w_{10},w_{26}\right>$ & $15$ & $\left<w_6,w_{10},w_{26}\right>$ & $5$\\
 & $\left<w_{15},w_{39}\right>$ & $14$ & $\left<w_6,w_{10},w_{26}\right>$ & $5$ & & $\left<w_6,w_{65}\right>$ & $15$ & $\left<w_6,w_{10},w_{26}\right>$ & $5$\\
 & $\left<w_{10},w_{39}\right>$ & $14$ & $\left<w_6,w_{10},w_{26}\right>$ & $5$ & & $\left<w_{15},w_{26}\right>$ & $15$ & $\left<w_6,w_{10},w_{26}\right>$ & $5$\\
 & $\left<w_6,w_{130}\right>$ & $17$ & $\left<w_5,w_6,w_{26}\right>$ & $5$ & & $\left<w_{10},w_{78}\right>$ & $17$ & $\left<w_2,w_5,w_{39}\right>$ & $6$\\
 & $\left<w_{15},w_{78}\right>$ & $20$ & $\left<w_3,w_5,w_{26}\right>$ & $8$ & & $\left<w_{30},w_{39}\right>$ & $15$ & $\left<w_5,w_6,w_{26}\right>$ & $5$\\
 & $\left<w_{30},w_{65}\right>$ & $17$ & $\left<w_3,w_{10},w_{26}\right>$ & $6$ & & $\left<w_3,w_{130}\right>$ & $18$ & $\left<w_3,w_{10},w_{13}\right>$ & $7$\\
 & $\left<w_5,w_{78}\right>$ & $18$ & $\left<w_2,w_5,w_{39}\right>$ & $6$ & & $\left<w_{13},w_{30}\right>$ & $17$ & $\left<w_2,w_{13},w_{15}\right>$ & $6$\\
\hline
\caption*{}
\end{longtable}
\end{center}
\end{prop}

\begin{proof}
    These curves $X_\Delta(N)$ are neither hyperelliptic nor trigonal by \cite{FurumotoHasegawa1999, HasegawaShimura2006}. Therefore, their $\C$-gonality is at least $4$.

    Suppose that there is a degree $4$ morphism from $X_0(N)/W_N$ to $\mathbb{P}^1$ for some pair $(N,W_N)$ from this table. We apply the Castelnuovo-Severi inequality with the degree $2$ quotient map $X_0(N)/W_N\to X_0(N)/W_N'$ and this hypothetical degree $4$ morphism. We can easily check that for all these curves we have \[g>2\cdot g'+2\cdot0+1\cdot3=2\cdot g'+3.\]
    Therefore, this degree $4$ morphism to $\PP^1$ would have to factor through $X_0(N)/W_N'$ and the curve $X_0(N)/W_N'$ would have to be hyperelliptic. However, \cite{FurumotoHasegawa1999, HasegawaShimura2006} tell us that none of these curves $X_0(N)/W_N'$ are hyperelliptic and we get a contradiction.
\end{proof}

\section{Rational morphisms to $\mathbb{P}^1$}\label{rationalmapsection}
In previous sections we were giving lower bounds on $\C$ and $\Q$-gonality. Now we give upper bounds by giving rational morphisms from $X_0(N)/W_N$ to $\mathbb{P}^1$.

\begin{prop}\label{genus4trig}
    The following genus $4$ curves $X_0(N)/W_N$ are trigonal over $\Q$: 
    \begin{align*}
        (N,W_N)\in\{&(130,\left<w_5,w_{13}\right>),(132,\left<w_4,w_{33}\right>),(154,\left<w_2,w_{77}\right>),(170,\left<w_5,w_{34}\right>),\\
        &(182,\left<w_2,w_{91}\right>),(186,\left<w_3,w_{62}\right>),(210,\left<w_2,w_{15},w_{21}\right>),\\
        &(255,\left<w_5,w_{51}\right>),(285,\left<w_3,w_{95}\right>),(286,\left<w_2,w_{143}\right>)\}.
    \end{align*}
\end{prop}

\begin{proof}
    For all these curves we used the \texttt{Magma} function Genus4GonalMap() to obtain a degree $3$ morphism to $\PP^1$ (it exists due to Proposition \ref{poonen} (v)). This morphism is defined over $\Q$, proving that these curves are $\Q$-trigonal.
\end{proof}

\begin{prop}\label{genus4gon4}
    The following genus $4$ curves $X_0(N)/W_N$ are not trigonal over $\Q$:
    \newpage
    \begin{center}
\begin{longtable}{|c|c|}

\hline
\addtocounter{table}{-1}
$N$ & $W_N$\\
\hline
$102$ & $\left<w_2,w_3\right>$, $\left<w_3,w_{34}\right>$\\
\hline
$114$ & $\left<w_2,w_3\right>$, $\left<w_3,w_{19}\right>$, $\left<w_3,w_{57}\right>$, $\left<w_6,w_{19}\right>$\\
\hline
$120$ & $\left<w_3,w_8\right>$, $\left<w_3,w_{40}\right>$\\
\hline
$126$ & $\left<w_2,w_7\right>$, $\left<w_7,w_{18}\right>$\\
\hline
$130$ & $\left<w_2,w_5\right>$, $\left<w_{10},w_{13}\right>$\\
\hline
$132$ & $\left<w_3,w_{11}\right>$, $\left<w_{12},w_{33}\right>$\\
\hline
$138$ & $\left<w_2,w_{69}\right>$\\
\hline
$140$ & $\left<w_5,w_7\right>$, $\left<w_5,w_{28}\right>$\\
\hline
$150$ & $\left<w_2,w_{25}\right>$, $\left<w_3,w_{25}\right>$, $\left<w_6,w_{25}\right>$\\
\hline
$154$ & $\left<w_7,w_{11}\right>$, $\left<w_7,w_{22}\right>$\\
\hline
$165$ & $\left<w_3,w_{11}\right>$, $\left<w_5,w_{11}\right>$\\
\hline
$168$ & $\left<w_3,w_{56}\right>$\\
\hline
$170$ & $\left<w_2,w_{85}\right>$\\
\hline
$174$ & $\left<w_6,w_{29}\right>$, $\left<w_6,w_{58}\right>$\\
\hline
$182$ & $\left<w_7,w_{26}\right>$, $\left<w_{13},w_{14}\right>$\\
\hline
$210$ & $\left<w_2,w_3,w_7\right>$, $\left<w_2,w_5,w_{21}\right>$, $\left<w_2,w_7,w_{15}\right>$, $\left<w_3,w_5,w_{14}\right>$\\
\hline
$220$ & $\left<w_5,w_{11}\right>$, $\left<w_{20},w_{44}\right>$, $\left<w_4,w_{55}\right>$\\
\hline
$222$ & $\left<w_2,w_{111}\right>$\\
\hline
$231$ & $\left<w_7,w_{33}\right>$, $\left<w_{21},w_{33}\right>$, $\left<w_{11},w_{21}\right>$\\
\hline
$330$ & $\left<w_6,w_{10},w_{11}\right>$\\
\hline
\caption*{}
\end{longtable}
\end{center}
\end{prop}

\begin{proof}
    We again used the \texttt{Magma} function Genus4GonalMap() to obtain a degree $3$ morphism to $\PP^1$. If the curve is $\Q$-trigonal, the given morphism will be defined over $\Q$. Otherwise, the function gives a morphism over a quadratic or a biquadratic field \cite{magma} (even though the trigonal map can always be defined over a number field of degree $\leq2$). 
    
    For all these curves the degree $3$ morphism is defined over a quadratic field, implying that they are not $\Q$-trigonal.
\end{proof}

\begin{prop}\label{genus6gon4}
    The quotient curve $X_0(N)/W_N$ has $\Q$-gonality equal to $4$ for \begin{align*}
        (N,W_N)\in\{&(234,\left<w_9,w_{26}\right>),(240,\left<w_{15},w_{48}\right>), (282,\left<w_3,w_{47}\right>), \\
        &(282,\left<w_6,w_{47}\right>), (310,\left<w_5,w_{31}\right>), (312,\left<w_{24},w_{39}\right>) \}.
    \end{align*}
\end{prop}

\begin{proof}
    These curves are of genus $6$. We use the \texttt{Magma} function Genus6GonalMap() to obtain a degree $4$ morphism to $\PP^1$ (it exists due to Proposition \ref{poonen} (v)). This morphism is defined over $\Q$ for all these curves, proving that these curves are $\Q$-tetragonal.
\end{proof}

\begin{prop}\label{deg2star}
    The quotient curve $X_0(N)/W_N$ has $\Q$-gonality at most $4$ for \begin{align*}
        N\in\{&120,126,132,138,140,150,154,156,165,168,170,174,180,182,186,190,195,198,204,\\
        &210,220,222,230,231,238,240,252,255,266,276,285,286,315,330,357,380,390\}.
    \end{align*} and all groups $W_N\subset B(N)$ of order $2^{\omega(N)-1}$.
\end{prop}

\begin{proof}
    For these levels $N$ the quotient curve $X_0^*(N)$ has $\Q$-gonality equal to $2$. For $N\neq252,315$ the genus of $X_0^*(N)$ is $1$ or $2$, and for $N=252,315$ the genus is equal to $3$, but $X_0^*(N)$ is hyperelliptic due to \cite{Hasegawa1997}.

    Since we have a degree $2$ quotient map $X_0(N)/W_N\to X_0^*(N)$, the composition map \[X_0(N)/W_N\to X_0^*(N)\to\PP^1\] is a degree $4$ rational morphism to $\PP^1$.
\end{proof}

\begin{prop}\label{deg4map210}
    The quotient curve $X_0(210)/W$ has $\Q$-gonality equal to $4$ for \[W\in\{\left<w_6,w_{10}\right>, \left<w_6,w_{14}\right>, \left<w_{10},w_{14}\right>, \left<w_{15},w_{21}\right>, \left<w_6,w_{35}\right>, \left<w_{10},w_{21}\right>, \left<w_{14},w_{15}\right>\}.\]
\end{prop}

\begin{proof}
    The curve $X_0(210)/\left<w_6,w_{10},w_{14}\right>$ is a genus $2$ hyperelliptic curve. We have a degree $2$ quotient map $X_0(210)/W\to X_0(210)/\left<w_6,w_{10},w_{14}\right>$. Thus the composition map \[X_0(210)/W\to X_0(210)/\left<w_6,w_{10},w_{14}\right>\to\PP^1\] is a degree $4$ rational morphism to $\PP^1$.
\end{proof}

\begin{prop}\label{pointsearch}
    The quotient curve $X_0(N)/W_N$ has $\Q$-gonality equal to $4$ for the following pairs $(N,W_N)$. Here $g$ denotes the genus of $X_0(N)/W_N$.
    \begin{center}
\begin{longtable}{|c|c|c||c|c|c||c|c|c|}

\hline
\addtocounter{table}{-1}
$N$ & $W_N$ & $g$ & $N$ & $W_N$ & $g$ & $N$ & $W_N$ & $g$\\
\hline
$210$ & $\left<w_{30},w_{35}\right>$ & $8$ & $234$ & $\left<w_{18},w_{26}\right>$ & $7$ & $240$ & $\left<w_3,w_{80}\right>$ & $7$\\
\hline
$273$ & $\left<w_3,w_{91}\right>$ & $7$ & $282$ & $\left<w_2,w_{47}\right>$ & $7$ & $294$ & $\left<w_6,w_{98}\right>$ & $7$\\
\hline
$306$ & $\left<w_9,w_{34}\right>$ & $10$ & $336$ & $\left<w_{16},w_{21}\right>$ & $12$ & $360$ & $\left<w_9,w_{40}\right>$ & $13$\\
\hline
\caption*{}
\end{longtable}
\end{center}
\end{prop}

\begin{proof}
    For all these curves we used \texttt{Magma} we found a rational effective divisor that is a sum of $4$ rational points and has Riemann-Roch dimension equal to $2$.
    
    When $g\leq8$, we searched for rational points using the \texttt{Magma} function PointSearch(). These computations were very fast. They took up to $10$ minutes.
    
    For $g\geq10$ the genus was too large and we were unable to use this function. Hence, we wrote a code that searched for rational points up to a certain height.

    We then computed Riemann-Roch dimensions of divisors constructed from these rational points and found a divisor of dimension $2$. The computation time for $N=336$ was around $10$ minutes, for $N=306$ around $12$ hours, and for $N=360$ around $3$ days.
\end{proof}

\section{Isomorphic curves}\label{section_isomorphisms}
In this section we present some isomorphisms between quotient curves $X_0(N)/W_N$ and use them to determine the gonality of these curves.

\begin{prop}[{\cite[Proposition 4]{FurumotoHasegawa1999}}]\label{prop4n}
    Let $N=4M$ with $M$ being odd. Let $W'$ be a subgroup of $B(N)$ generated by $w_4,w_{M_1},\ldots,W_{M_s} \ (M_i\mid\mid M)$. Then we have the following isomorphism \[X_0(N)/W'\cong X_0(2M)/\left<w_{M_i}\right>.\] This isomorphism is defined over $\Q$, see \cite[Proposition 7]{Hasegawa1997}.
\end{prop}

\begin{prop}\label{prop4n_gon4}
    The quotient curve $X_0(N)/W_N$ has $\Q$-gonality equal to $4$ for the following pairs $(N,W_N)$.

    \begin{center}
\begin{longtable}{|c|c|c||c|c|c|}

\hline
\addtocounter{table}{-1}
$N$ & $W_N$ & $Y$ & $N$ & $W_N$ & $Y$\\
\hline
$228$ & $\left<w_3,w_4\right>$ & $X_0(114)/w_3$ & $228$ & $\left<w_4,w_{19}\right>$ & $X_0(114)/w_{19}$\\
\hline
$260$ & $\left<w_4,w_{13}\right>$ & $X_0(130)/w_{13}$ & $260$ & $\left<w_4,w_{65}\right>$ & $X_0(130)/w_{65}$\\
\hline
$300$ & $\left<w_4,w_{75}\right>$ & $X_0(150)/w_{75}$ & & &\\
\hline
\caption*{}
\end{longtable}
\end{center}
\end{prop}

\begin{proof}
    We have the isomorphism $X_0(N)/W_N\cong Y$ defined over $\Q$ by Proposition \ref{prop4n} and we know from \cite[Proposition 3.2]{Orlic2024} that these curves $Y$ have $\Q$-gonality equal to $4$.
\end{proof}

\begin{prop}\label{prop4n_not_tetragonal}
    The quotient curve $X_0(N)/W_N$ has $\C$-gonality at least $5$ for the following pairs $(N,W_N)$.

    \begin{center}
\begin{longtable}{|c|c|c||c|c|c|}

\hline
\addtocounter{table}{-1}
$N$ & $W_N$ & $Y$ & $N$ & $W_N$ & $Y$\\
\hline
$228$ & $\left<w_4,w_{57}\right>$ & $X_0(114)/w_{57}$ & $260$ & $\left<w_4,w_5\right>$ & $X_0(130)/w_5$\\
\hline
$300$ & $\left<w_3,w_4\right>$ & $X_0(150)/w_3$ & $300$ & $\left<w_4,w_{25}\right>$ & $X_0(150)/w_{25}$\\
\hline
$340$ & $\left<w_4,w_5\right>$ & $X_0(170)/w_5$ & $340$ & $\left<w_4,w_{17}\right>$ & $X_0(170)/w_{17}$\\
\hline
$340$ & $\left<w_4,w_{85}\right>$ & $X_0(170)/w_{85}$ & $348$ & $\left<w_3,w_4\right>$ & $X_0(174)/w_3$\\
\hline
$348$ & $\left<w_4,w_{29}\right>$ & $X_0(174)/w_{29}$ & $348$ & $\left<w_4,w_{87}\right>$ & $X_0(174)/w_{87}$\\
\hline
$364$ & $\left<w_4,w_7\right>$ & $X_0(182)/w_7$ & $364$ & $\left<w_4,w_{13}\right>$ & $X_0(182)/w_{13}$\\
\hline
$364$ & $\left<w_4,w_{91}\right>$ & $X_0(182)/w_{91}$ & $372$ & $\left<w_3,w_4\right>$ & $X_0(186)/w_3$\\
\hline
$372$ & $\left<w_4,w_{31}\right>$ & $X_0(186)/w_{31}$ & $372$ & $\left<w_4,w_{93}\right>$ & $X_0(186)/w_{93}$\\
\hline
$444$ & $\left<w_3,w_4\right>$ & $X_0(222)/w_3$ & $444$ & $\left<w_4,w_{37}\right>$ & $X_0(222)/w_{37}$\\
\hline
$444$ & $\left<w_4,w_{111}\right>$ & $X_0(222)/w_{111}$ & $460$ & $\left<w_4,w_5\right>$ & $X_0(230)/w_5$\\
\hline
$460$ & $\left<w_4,w_{23}\right>$ & $X_0(230)/w_{23}$ & $460$ & $\left<w_4,w_{115}\right>$ & $X_0(230)/w_{115}$\\
\hline
$532$ & $\left<w_4,w_7\right>$ & $X_0(266)/w_7$ & $532$ & $\left<w_4,w_{19}\right>$ & $X_0(266)/w_{19}$\\
\hline
$532$ & $\left<w_4,w_{133}\right>$ & $X_0(266)/w_{133}$ & $564$ & $\left<w_3,w_4\right>$ & $X_0(282)/w_3$\\
\hline
$564$ & $\left<w_4,w_{47}\right>$ & $X_0(282)/w_{47}$ & $564$ & $\left<w_4,w_{141}\right>$ & $X_0(282)/w_{141}$\\
\hline
$572$ & $\left<w_4,w_{11}\right>$ & $X_0(286)/w_{11}$ & $572$ & $\left<w_4,w_{13}\right>$ & $X_0(286)/w_{13}$\\
\hline
$572$ & $\left<w_4,w_{143}\right>$ & $X_0(286)/w_{143}$ & $620$ & $\left<w_4,w_5\right>$ & $X_0(310)/w_5$\\
\hline
$620$ & $\left<w_4,w_{31}\right>$ & $X_0(310)/w_{31}$ & $620$ & $\left<w_4,w_{155}\right>$ & $X_0(310)/w_{155}$\\
\hline
$660$ & $\left<w_3,w_4,w_5\right>$ & $X_0(330)/\left<w_3,w_5\right>$ & $660$ & $\left<w_3,w_4,w_{11}\right>$ & $X_0(330)/\left<w_3,w_{11}\right>$\\
\hline
$660$ & $\left<w_4,w_5,w_{11}\right>$ & $X_0(330)/\left<w_5,w_{11}\right>$ & $660$ & $\left<w_3,w_4,w_{55}\right>$ & $X_0(330)/\left<w_3,w_{55}\right>$\\
\hline
$660$ & $\left<w_4,w_5,w_{33}\right>$ & $X_0(330)/\left<w_5,w_{33}\right>$ & $660$ & $\left<w_4,w_{11},w_{15}\right>$ & $X_0(330)/\left<w_{11},w_{15}\right>$\\
\hline
$660$ & $\left<w_4,w_{15},w_{33}\right>$ & $X_0(330)/\left<w_{15},w_{33}\right>$ & $780$ & $\left<w_3,w_4,w_5\right>$ & $X_0(390)/\left<w_3,w_5\right>$\\
\hline
$780$ & $\left<w_3,w_4,w_{13}\right>$ & $X_0(390)/\left<w_3,w_{13}\right>$ & $780$ & $\left<w_4,w_5,w_{13}\right>$ & $X_0(390)/\left<w_5,w_{13}\right>$\\
\hline
$780$ & $\left<w_3,w_4,w_{65}\right>$ & $X_0(390)/\left<w_3,w_{65}\right>$ & $780$ & $\left<w_4,w_5,w_{39}\right>$ & $X_0(390)/\left<w_5,w_{39}\right>$\\
\hline
$780$ & $\left<w_4,w_{13},w_{15}\right>$ & $X_0(390)/\left<w_{13},w_{15}\right>$ & $780$ & $\left<w_4,w_{15},w_{39}\right>$ & $X_0(390)/\left<w_{15},w_{39}\right>$\\
\hline
\caption*{}
\end{longtable}
\end{center}
\end{prop}

\begin{proof}
    We have the isomorphism $X_0(N)/W_N\cong Y$ defined over $\Q$ by Proposition \ref{prop4n}.
    
    For curves $Y=X_0^{+d}(2M)$ we know from \cite{Orlic2024} that they have $\C$-gonality greater than $4$. For curves $Y=X_0(2M)/\left<w_{d_1},w_{d_2}\right>$ we gave a lower bound on $\C$-gonality in Propositions \ref{Fp_gonality}, \ref{CSprop_other}.
\end{proof}

\begin{prop}[{\cite[Proposition 5]{FurumotoHasegawa1999}}]\label{iso9prop}
    Assume that $9\mid\mid N$. Let $W'$ be a subgroup of $B(N)$ generated by $w_{N_1},\ldots,w_{N_s} \ (N_i\mid\mid N)$, and let $W''=\left<w_9^{\epsilon(N_i)}w_{N_i}\right>$, where \[\epsilon(M)=\begin{cases}
        0 & \textup{ if } M\equiv1\pmod{3} \textup{ or if } 9\mid\mid M \textup{ and } M/9\equiv1\pmod{3},\\
        1 & \textup{ otherwise}.
    \end{cases}\] Then we have the following isomorphism \[X_0(N)/W'\cong X_0(N)/W''.\] It is defined over $\Q(\sqrt{-3})$, see \cite[Lemma 1]{FurumotoHasegawa1999} and \cite[Proposition 4.19]{bars22biellipticquotients}.
\end{prop}

Notice that Proposition \ref{iso9prop} can give a non-trivial isomorphism only when $w_9\notin W'$ (if $w_9\in W'$, then $W'=W''$ by definition).

\begin{prop}\label{prop9n_not_tetragonal}
    The quotient curve $X_0(N)/W_N$ has $\C$-gonality at least $5$ for the following pairs $(N,W_N)$.

    \begin{center}
\begin{longtable}{|c|c|c||c|c|c|}

\hline
\addtocounter{table}{-1}
$N$ & $W_N$ & $W''$ & $N$ & $W_N$ & $W''$\\
\hline
$306$ & $\left<w_{18},w_{34}\right>$ & $\left<w_2,w_{17}\right>$ & $306$ & $\left<w_2,w_{153}\right>$ & $\left<w_{17},w_{18}\right>$\\
\hline
$342$ & $\left<w_{18},w_{19}\right>$ & $\left<w_2,w_{19}\right>$ & $342$ & $\left<w_2,w_{171}\right>$ & $\left<w_{18},w_{38}\right>$\\
\hline
$360$ & $\left<w_{40},w_{45}\right>$ & $\left<w_5,w_8\right>$ & $360$ & $\left<w_5,w_{72}\right>$ & $\left<w_8,w_{45}\right>$\\
\hline
$396$ & $\left<w_{11},w_{36}\right>$ & $\left<w_{36},w_{44}\right>$ & $396$ & $\left<w_4,w_{99}\right>$ & $\left<w_4,w_{11}\right>$\\
\hline
$414$ & $\left<w_{18},w_{46}\right>$ & $\left<w_2,w_{23}\right>$ & $414$ & $\left<w_2,w_{207}\right>$ & $\left<w_{18},w_{23}\right>$\\
\hline
$495$ & $\left<w_{45},w_{55}\right>$ & $\left<w_5,w_{11}\right>$ & $495$ & $\left<w_5,w_{99}\right>$ & $\left<w_{11},w_{45}\right>$\\
\hline
$504$ & $\left<w_8,w_{63}\right>$ & $\left<w_{56},w_{63}\right>$ & $504$ & $\left<w_7,w_{72}\right>$ & $\left<w_7,w_8\right>$\\
\hline
$558$ & $\left<w_{18},w_{31}\right>$ & $\left<w_2,w_{31}\right>$ & $558$ & $\left<w_2,w_{279}\right>$ & $\left<w_{18},w_{62}\right>$\\
\hline
$585$ & $\left<w_{13},w_{45}\right>$ & $\left<w_5,w_{13}\right>$ & $585$ & $\left<w_5,w_{117}\right>$ & $\left<w_{45},w_{65}\right>$\\
\hline
$630$ & $\left<w_5,w_7,w_{18}\right>$ & $\left<w_2,w_7,w_{45}\right>$ & $630$ & $\left<w_5,w_{14},w_{18}\right>$ & $\left<w_2,w_{35},w_{45}\right>$\\
\hline
$630$ & $\left<w_7,w_{10},w_{18}\right>$ & $\left<w_2,w_5,w_7\right>$ & $630$ & $\left<w_{10},w_{14},w_{18}\right>$ & $\left<w_2,w_5,w_{63}\right>$\\
\hline
$990$ & $\left<w_2,w_{45},w_{55}\right>$ & $\left<w_5,w_{11},w_{18}\right>$ & $990$ & $\left<w_5,w_{18},w_{22}\right>$ & $\left<w_2,w_{11},w_{45}\right>$\\
\hline
$990$ & $\left<w_{10},w_{11},w_{18}\right>$ & $\left<w_2,w_5,w_{99}\right>$ & $990$ & $\left<w_{10},w_{18},w_{22}\right>$ & $\left<w_2,w_5,w_{11}\right>$\\
\hline
\caption*{}
\end{longtable}
\end{center}
\end{prop}

\begin{proof}
    We have the isomorphism $X_0(N)/W_N\cong X_0(N)/W''$ defined over $\C$ by Proposition \ref{prop4n}.
    For curves $X_0(N)/W''$ of genus $\leq9$ we know from Proposition \ref{betti0prop} that they are not $\C$-tetragonal. For curves of genus $\geq10$ we know from Propositions \ref{Fp_gonality}, \ref{prop_counting_Fpn_points}, \ref{CSprop_star} that they are not $\Q$-tetragonal. Hence they are also not $\C$-tetragonal by Corollary \ref{towerthmcor}.
\end{proof}

\section{Proofs of the main theorems}\label{thmproof_section}

The proof of Theorem \ref{trigonalthm} is given in Propositions \ref{genus4trig} and \ref{genus4gon4} where we determine the fields of definition of trigonal maps. We now give the proof of Theorem \ref{tetragonalthm}.

\begin{proof}[Proof of Theorem \ref{tetragonalthm}]
    If $\omega(N)\leq2$, then there are no subgroups $W_N\subseteq B(N)$ such that $4\leq |W_N|\leq2^{\omega(N)-1}$. Thus we may suppose that $\omega(N)\geq3$. We may eliminate all levels $N$ such that $\gon_\C(X_0^*(N))\geq5$ by Proposition \ref{poonen} (vii). Furthermore, in Proposition \ref{Ctetragonal_star_genus6} we deal with the levels $N$ such that $\gon_\C(X_0^*(N))=4<\gon_\Q(X_0^*(N))$. This leaves us with levels $N$ for which $\gon_\Q(X_0^*(N))\leq4$ (and $N=378$). All such levels are listed in \cite{Hasegawa1997}, \cite{HasegawaShimura2000}, \cite{Orlic2025star}.

    For all these levels $N$ we have $\omega(N)\leq4$. For curves $X_0(N)/W_N$ listed in the theorem we use Propositions \ref{genus6gon4}, \ref{betti_positive_prop}, \ref{deg2star}, \ref{pointsearch}, \ref{prop4n_gon4} to prove the existence of a rational degree $4$ morphism to $\PP^1$.
    
    For other curves we use Propositions \ref{Fp_gonality}, \ref{Fp2points}, \ref{betti0prop}, \ref{CSprop_star}, \ref{CSprop_other}, \ref{prop4n_not_tetragonal}, \ref{prop9n_not_tetragonal} to disprove the existence of a degree $4$ morphism to $\PP^1$. Note that Propositions \ref{Fp_gonality}, \ref{Fp2points} only give a lower bound on the $\Q$-gonality, but all curves there are of genus $g\geq10$ and we can use Corollary \ref{towerthmcor} to obtain a bound on the $\C$-gonality.

    For \[N\in\{630,660,714,770,780,798,910,990\}\] we only considered curves $X_0(N)/W_N$ with $|W_N|=8$. This is because we proved that all such curves have $\C$-gonality at least $5$. Proposition \ref{poonen} (vii) then tells us that all curves $X_0(N)/W_N$ with $|W_N|=4$ also have $\C$-gonality at least $5$ (as every group $W_N$ of order $4$ is a subgroup of some group of order $8$).
\end{proof}

For the end of the paper we summarize the result of Theorem \ref{tetragonalthm}. On the following page we present the table containing all $\Q$-tetragonal (and $\C$-tetragonal) curves $X_0(N)/W_N$ with $4\leq|W_N|\leq 2^{\omega(N)-1}$.

\newpage

\begin{center}
\begin{longtable}{|c|c|}
\caption{All $\Q$-tetragonal curves $X_0(N)/W_N$ with $4\leq|W_N|\leq 2^{\omega(N)-1}$.}
\label{main:table}\\
\hline
$N$ & $W_N$\\
    \hline
$102$ & $\left<w_2,w_3\right>$, $\left<w_3,w_{34}\right>$\\
\hline
$114$ & $\left<w_2,w_3\right>$, $\left<w_3,w_{19}\right>$, $\left<w_3,w_{57}\right>$, $\left<w_6,w_{19}\right>$\\
\hline
$120$ & $\left<w_3,w_8\right>$, $\left<w_3,w_{40}\right>$, $\left<w_5,w_8\right>$\\
\hline
$126$ & $\left<w_2,w_7\right>$, $\left<w_2,w_9\right>$, $\left<w_7,w_{18}\right>$\\
\hline
$130$ & $\left<w_2,w_5\right>$, $\left<w_{10},w_{13}\right>$\\
\hline
$132$ & $\left<w_3,w_4\right>$, $\left<w_3,w_{11}\right>$, $\left<w_{12},w_{33}\right>$\\
\hline
$138$ & $\left<w_2,w_3\right>$, $\left<w_6,w_{46}\right>$, $\left<w_3,w_{46}\right>$, $\left<w_2,w_{69}\right>$\\
\hline
$140$ & $\left<w_4,w_5\right>$, $\left<w_4,w_7\right>$, $\left<w_5,w_7\right>$, $\left<w_5,w_{28}\right>$\\
\hline
$150$ & $\left<w_2,w_3\right>$, $\left<w_2,w_{25}\right>$, $\left<w_3,w_{25}\right>$, $\left<w_6,w_{25}\right>$\\
\hline
$154$ & \textup{all of size 4 except} $\left<w_2,w_{77}\right>$, $\left<w_{11},w_{14}\right>$\\
\hline
$156$ & $\left<w_3,w_4\right>$, $\left<w_4,w_{13}\right>$, $\left<w_{12},w_{13}\right>$, $\left<w_3,w_{52}\right>$\\
\hline
$165$ & \textup{all of size 4 except} $\left<w_5,w_{11}\right>$\\
\hline
$168$ & \textup{all of size 4 except} $\left<w_{24},w_{56}\right>$\\
\hline
$170$ & \textup{all of size 4 except} $\left<w_{10},w_{17}\right>$, $\left<w_5,w_{34}\right>$\\
\hline
$174$ & \textup{all of size 4 except} $\left<w_3,w_{29}\right>$, $\left<w_2,w_{87}\right>$\\
\hline
$180$ & \textup{all of size 4}\\
\hline
$182$ & \textup{all of size 4 except} $\left<w_2,w_{91}\right>$, $\left<w_{14},w_{26}\right>$\\
\hline
$186$ & \textup{all of size 4 except} $\left<w_3,w_{62}\right>$\\
\hline
$190$ & $\left<w_3,w_5\right>$, $\left<w_3,w_{13}\right>$, $\left<w_5,w_{13}\right>$, $\left<w_{13},w_{15}\right>$\\
\hline
$198$ & \textup{all of size 4}\\
\hline
$204$ & \textup{all of size 4 except} $\left<w_3,w_{68}\right>$\\
\hline
$210$ & $\left<w_6,w_{10}\right>$, $\left<w_6,w_{14}\right>$, $\left<w_{10},w_{14}\right>$, $\left<w_{15},w_{21}\right>$, $\left<w_6,w_{35}\right>$, $\left<w_{10},w_{21}\right>$, $\left<w_{14},w_{15}\right>$, $\left<w_{30},w_{35}\right>$\\
& $\left<w_2,w_3,w_5\right>$, $\left<w_2,w_3,w_7\right>$, $\left<w_2, w_5,w_{21}\right>$, $\left<w_2,w_7,w_{15}\right>$, $\left<w_3,w_7,w_{10}\right>$, $\left<w_3,w_5,w_{14}\right>$\\
\hline
$220$ & \textup{all of size 4}\\
\hline
$222$ & \textup{all of size 4 except} $\left<w_6,w_{74}\right>$\\
\hline
$228$ & $\left<w_3,w_4\right>$, $\left<w_4,w_{19}\right>$, $\left<w_3,w_{76}\right>$\\
\hline
$230$ & \textup{all of size 4}\\
\hline
$231$ & \textup{all of size 4 except} $\left<w_3,w_{77}\right>$\\
\hline
$234$ & $\left<w_{18},w_{26}\right>$, $\left<w_9,w_{26}\right>$, $\left<w_2,w_{117}\right>$\\
\hline
$238$ & $\left<w_2,w_7\right>$, $\left<w_2,w_{17}\right>$, $\left<w_7,w_{34}\right>$, $\left<w_{14},w_{17}\right>$\\
\hline
$240$ & $\left<w_5,w_{16}\right>$, $\left<w_{15},w_{16}\right>$, $\left<w_{15},w_{48}\right>$, $\left<w_3,w_{80}\right>$\\
\hline
$252$ & \textup{all of size 4}\\
\hline
$255$ & \textup{all of size 4 except} $\left<w_5,w_{51}\right>$\\
\hline
$258$ & $\left<w_2,w_{43}\right>$, $\left<w_6,w_{86}\right>$\\
\hline
$260$ & $\left<w_4,w_{13}\right>$, $\left<w_4,w_{65}\right>$\\
\hline
$264$ & $\left<w_8,w_{33}\right>$, $\left<w_{11},w_{24}\right>$\\
\hline
$266$ & \textup{all of size 4}\\
\hline
$273$ & $\left<w_3,w_7\right>$, $\left<w_3,w_{13}\right>$, $\left<w_3,w_{91}\right>$\\
\hline
$276$ & \textup{all of size 4}\\
\hline
$282$ & $\left<w_2,w_{47}\right>$, $\left<w_3,w_{47}\right>$, $\left<w_6,w_{47}\right>$\\
\hline
$285$ & \textup{all of size 4 except} $\left<w_3,w_{95}\right>$\\
\hline
$286$ & \textup{all of size 4 except} $\left<w_2,w_{143}\right>$\\
\hline
$294$ & $\left<w_6,w_{98}\right>$\\
\hline
$300$ & $\left<w_4,w_{75}\right>$\\
\hline
$310$ & $\left<w_5,w_{31}\right>$\\
\hline
$312$ & $\left<w_{24},w_{39}\right>$\\
\hline
$315$ & \textup{all of size 4}\\
\hline
$330$ & \textup{all of size 8 except} $\left<w_3,w_{10},w_{11}\right>$\\
\hline
$336$ & $\left<w_{16},w_{21}\right>$\\
\hline
$357$ & \textup{all of size 4}\\
\hline
$360$ & $\left<w_9,w_{40}\right>$\\
\hline
$380$ & \textup{all of size 4}\\
\hline
$390$ & \textup{all of size 8}\\
\hline
\end{longtable}
\end{center}

\bibliographystyle{siam}
\bibliography{references}

\end{document}